\documentclass[12pt]{amsart}
\usepackage{amscd,amssymb,amsthm,amsmath,amssymb,mathrsfs,enumerate,ulem}
\usepackage[matrix,arrow,curve]{xy}
\usepackage[margin=2cm]{geometry}
\usepackage{hyperref}

\newtheorem{theorem}{Theorem}
\newtheorem{proposition}[theorem]{Proposition}
\newtheorem{lemma}[theorem]{Lemma}
\newtheorem*{lemma*}{Lemma}

\newtheorem*{maintheorem*}{Main Theorem}
\newtheorem*{corollary*}{Corollary}
\newtheorem*{conjecture*}{Conjecture}
\newtheorem*{theorem*}{Theorem}
\newtheorem*{question*}{Question}
\newtheorem*{proposition*}{Proposition}
\newtheorem*{remark*}{Remark}

\theoremstyle{definition}
\newtheorem{example}[theorem]{Example}
\newtheorem*{example*}{Example}
\newtheorem{definition}[theorem]{Definition}

\theoremstyle{remark}

\newcommand{\C}{\mathbb{C}}
\newcommand{\PP}{\mathbb{P}}

\makeatletter\@addtoreset{equation}{section} \makeatother

\author{Hamid Abban, Paolo Cascini, Ivan Cheltsov}

\title{A Matsushima theorem for K-polystable polarised smooth Fano threefolds}

\let\origmaketitle\maketitle
\def\maketitle{
  \begingroup
  \def\uppercasenonmath##1{} 
  \let\MakeUppercase\relax 
  \origmaketitle
  \endgroup
}

\begin{document}

\begin{abstract}
We prove that if $X$ is a smooth Fano threefold and $L$ is an ample $\mathbb{Q}$-divisor such that $(X,L)$ is K-polystable,
then the automorphism group $\operatorname{Aut}(X)$ is reductive.
This verifies the reductivity statement predicted by the Yau--Tian--Donaldson conjecture in the setting of smooth Fano threefolds
with arbitrary ample polarisation.
\end{abstract}

\address{\textit{Hamid Abban}\newline
\textnormal{University of Nottingham, Nottingham, England
\newline
\texttt{hamid.abban@nottingham.ac.uk}}}

\address{\textit{Paolo Cascini}\newline
\textnormal{Imperial College, London, England
\newline
\texttt{p.cascini@imperial.ac.uk}}}

\address{ \textit{Ivan Cheltsov}\newline
\textnormal{University of Edinburgh, Edinburgh, Scotland
\newline
\texttt{i.cheltsov@ed.ac.uk}}}

\maketitle

\section{Introduction}
\label{section:intro}

Let $X$ be a projective normal variety and let $L$ be an ample $\mathbb{Q}$-divisor on $X$.
If $X$ is smooth, the famous Yau--Tian--Donaldson conjecture states that the pair $(X,L)$ is K-polystable if and only if
$X$ admits a constant scalar curvature K\"ahler (cscK) metric in $c_1(L)$.
This conjecture holds in two notable cases: toric surfaces \cite{CodogniStoppa,Donaldson2002,Donaldson2009}
and Fano varieties polarised by their anticanonical divisors~\cite{CDS}. Although the K-polystability of the pair $(X,L)$ is a necessary condition for the existence of a cscK metric \cite{BDL16,Stoppa}, and recent progress has been made in the general case (see for example \cite{BJ,Li2022}), the conjecture remains elusive. On the other hand,
the non-reductivity of the automorphism group is an obstruction for the existence of such metrics \cite{Mat57}, which is the focus of this article in relation to K-polystability: if the variety $X$ is smooth and the pair $(X,L)$ is K-polystable, then the Yau--Tian--Donaldson conjecture implies that $\mathrm{Aut}(X)$ is reductive.
Note that Alper, Blum, Halpern-Leistner and Xu proved that K-polystability of a Fano variety $(X,-K_X)$ implies reductivity of $\mathrm{Aut}(X)$ \cite{AlperBlumHalpernLeistnerXu}.
It is expected that the group $\mathrm{Aut}(X)$ is reductive if $(X,L)$ is K-polystable for some ample $\mathbb{Q}$-divisor $L$ on $X$. The goal of this paper is to verify this for smooth Fano threefolds.

\begin{maintheorem*}
Let $X$ be a smooth Fano threefold, and let $L$ be an ample $\mathbb{Q}$-divisor on $X$.
If $(X,L)$ is K-polystable, then $\mathrm{Aut}(X)$ is reductive.
\end{maintheorem*}

Note that, in \cite{CheltsovMartinez-Garcia}, this assertion is verified for smooth two-dimensional Fano varieties (del Pezzo surfaces). For convenience, we will present a short proof of this fact in Example~\ref{example:F1-dP7}.

Smooth Fano threefolds have been classified into 105 families by Iskovskikh, Mori and Mukai.
In this paper, we follow the Mori--Mukai numbering of the 105 families, written as ``Family \textnumero $m.n$'',
in which $m$ is the rank of the Picard group of the threefold, ranging from 1 to 10, and $n$ is simply a list number.
For a description of these deformation families, we refer the reader to \cite{Fanography,IsPr99}.
Recall from \cite{CheltsovShramovPrzyjalkowski} that all smooth members of the following 22 Families have non-reductive automorphism groups:
\begin{center}
Families \textnumero~2.28, \textnumero 2.30, \textnumero 2.31, \textnumero 2.33, \textnumero 2.35, \textnumero 2.36, \\
\textnumero 3.16, \textnumero 3.18, \textnumero 3.21, \textnumero 3.22, \textnumero 3.23, \textnumero 3.24, \textnumero 3.26, \textnumero 3.28,
\textnumero 3.29, \textnumero 3.30, \textnumero 3.31,\\  \textnumero 4.8, \textnumero 4.9, \textnumero 4.10, \textnumero 4.11, \textnumero 4.12.
\end{center}
All smooth members of these Families are K-unstable with respect to the anticanonical polarisation.
Note that Family \textnumero~2.28 has one-dimensional moduli, while each of the remaining 21 Families contains a unique smooth Fano threefold.
In addition, each family among the
\begin{center}
Families \textnumero 1.10, \textnumero 2.21, \textnumero 2.26, \textnumero 3.13
\end{center}
contains a unique smooth Fano threefold with non-reductive automorphism group. The unique smooth member of Family \textnumero 1.10 that has a non-reductive automorphism group is not K-polystable with respect to any polarisation as its Picard rank is one. This threefold
is simply ignored in the proof of the Main Theorem. Short descriptions of the remaining smooth Fano threefolds with non-reductive automorphism groups are given in Table~\ref{table:Fanos}.

\subsection*{Structure of the proof}

By the classification of smooth Fano threefolds with non-reductive automorphism group, it suffices to consider the families listed in Table~1.
For each such threefold~$X$, and for every ample $\mathbb{Q}$-divisor $L$ on $X$, we show that $(X,L)$ fails to be K-polystable.

The proof relies on two main destabilisation mechanisms, together with explicit birational constructions adapted to the geometry of each family.

\medskip

\noindent
\textbf{(1) Destabilisation via test configurations and Futaki invariants.}
In several cases, we use explicit test configurations whose Donaldson--Futaki invariant is non-positive, and which are not product configurations. This shows that $(X,L)$ is not K-polystable.

This includes both degenerations to more symmetric central fibres and product test configurations induced by $\C^*$-actions. This method applies to the families
\[
\text{\textnumero } 2.21,\;
\text{\textnumero } 2.26,\;
\text{\textnumero } 3.13,\text{and } \text{\textnumero } 3.16.
\]

These are covered in Section \ref{section:degeneration}.

\medskip

\noindent
\textbf{(2) Destabilisation via the $\beta$-invariant.}
In the remaining cases, we exhibit a geometrically natural prime divisor $F$ over $X$ such that $\beta_L(F) < 0$, which implies K-instability by Theorem \ref{theorem:DervanLegendre}. The divisor $F$ typically arises as an exceptional divisor of a blow-up, or as the strict transform of a distinguished surface appearing in the birational geometry of $X$.

This method applies to the families
\[
\text{\textnumero } 2.28,\;
\text{\textnumero } 3.18,\;
\text{\textnumero } 3.21,\;
\text{\textnumero } 3.22,\;
\text{\textnumero } 3.23,\;
\text{\textnumero } 3.24,\;
\text{\textnumero } 3.29,\;
\text{\textnumero } 4.8,\;
\text{\textnumero } 4.9,\;
\text{\textnumero } 4.10,\;
\text{\textnumero } 4.11,\;
\text{\textnumero } 4.12.
\]
These families are covered in Section~\ref{section:beta}.

\medskip

\noindent
\textbf{(3) Birational reduction and auxiliary extractions.}
In a number of families, the destabilising divisor is not immediately visible on $X$ itself, but appears naturally after passing to a suitable birational model. Using the explicit Mori--Mukai descriptions, we construct birational morphisms that reduce the computation of $\beta_L(F)$ to a simpler model.

This approach is used for the families
\[
\text{\textnumero } 2.30,\;
\text{\textnumero } 2.31,\;
\text{\textnumero } 2.33,\;
\text{\textnumero } 2.35,\;
\text{\textnumero } 2.36,\;
\text{\textnumero } 3.26,\;
\text{\textnumero } 3.28,\;
\text{\textnumero } 3.30,\;
\text{\textnumero } 3.31.
\]
These families are covered in Section~\ref{section:birational-reduction}.

\medskip

Combining these arguments, we show that every smooth Fano threefold with non-reductive automorphism group is K-unstable or at least not K-polystable for every ample polarisation. Therefore, if $(X,L)$ is K-polystable for some ample $\mathbb{Q}$-divisor $L$, then $\operatorname{Aut}(X)$ must be reductive.

\medskip

Roughly speaking, the families in Table~1 split into those handled by explicit $\beta$-invariant computations on the given model, those requiring an auxiliary extraction to reveal the destabilising divisor, and those ruled out by constructing suitable test configurations. The latter sections are organised accordingly, although the birational realisations differ from family to family.

\begin{table}
\label{table:Fanos}
\caption{Smooth Fano threefolds with non-reductive automorphism groups}
\begin{center}
\renewcommand\arraystretch{1.4}
\begin{tabular}{|c|c|c|}
\hline
 \textbf{Family} & \textbf{Short description} & \textbf{Section} \\
\hline
\hline
 \shortstack{\textnumero {2.21}\\{ }\\{ }}&  \shortstack{\\blow-up of the smooth quadric threefold\\ $Q\subset\mathbb{P}^4$ along a twisted quartic curve} &  \shortstack{\S~\ref{subsection:2-21}\\{ }\\{ }}\\
\hline
 {\textnumero 2.26} & blow-up of $Q\subset\mathbb{P}^4$ along a twisted cubic & \S~\ref{subsection:2-26}\\
\hline
\textnumero {2.28} & blow-up of $\mathbb{P}^3$ along a plane cubic curve & \S~\ref{subsection:2-28}\\
\hline
\textnumero {2.30} & blow-up of $\mathbb{P}^3$ along a conic  &\S~\ref{subsection:2-30} \\
\hline
\textnumero {2.31} & blow-up of $Q\subset\mathbb{P}^4$ along a line &   \S~\ref{subsection:2-31}\\
\hline
\textnumero {2.33}& blow-up of $\mathbb{P}^3$ in a~line  & \S~\ref{subsection:2-33}\\
\hline
  \textnumero {2.35} & $V_7=$ blow-up of $\mathbb{P}^3$ in a~point & \S~\ref{subsection:2-35}\\
\hline
  \textnumero {2.36} & $\mathbb{P}\big(\mathcal{O}_{\mathbb{P}^2}\oplus\mathcal{O}_{\mathbb{P}^2}(2)\big)$  & \S~\ref{subsection:2-36} \\
\hline
\shortstack{\\\textnumero {3.13}\\{ }\\{ }} & \shortstack{\\complete intersection in $\mathbb{P}^2\times\mathbb{P}^2\times\mathbb{P}^2$\\ of divisors of degree $(1,1,0)$, $(1,0,1)$, $(0,1,1)$} &\shortstack{\\\S~\ref{subsection:3-13}\\{ }\\{ }} \\
\hline
\shortstack{ {\textnumero 3.16}\\{ }\\{ }\\{ }\\{ }} & \shortstack{\\blow-up of $V_7$ along proper transform of\\ twisted cubic in $\mathbb{P}^3$ containing the blown-up point} &\shortstack{\S~\ref{subsection:3-16}\\{ }\\{ }}  \\
\hline
\textnumero {3.18} & blow-up of $\mathbb{P}^3$ along  union of line and conic & \S~\ref{subsection:3-18}\\
\hline
\textnumero {3.21}  & blow-up of $\mathbb{P}^1\times\mathbb{P}^2$ along curve of degree  $(2,1)$ & \S~\ref{subsection:3-21}\\
\hline
\shortstack{\\\textnumero {3.22}\\{ }\\{ }\\{ }} & \shortstack{\\blow-up of $\mathbb{P}^1\times\mathbb{P}^2$ along conic \\ in fibre of projection $\mathbb{P}^{1}\times\mathbb{P}^2\to\mathbb{P}^1$} & \shortstack{\\\S~\ref{subsection:3-22}\\{ }\\{ }\\{ }}\\
\hline
\shortstack{\textnumero {3.23}\\{ }\\{ }\\{ }\\{ }} & \shortstack{\\blow-up of $V_7$ along proper transform of smooth\\ conic in $\mathbb{P}^3$ passing through the blown-up point} & \shortstack{\S~\ref{subsection:3-23}\\{ }\\{ }}\\
\hline
\textnumero {3.24} & blow-up of $\mathbb{P}^1$-bundle $W\to\mathbb{P}^2$ along a fibre&\S~\ref{subsection:3-24}\\
\hline
  \textnumero {3.26} & blow-up of $\mathbb{P}^3$ in a~line and a~point &  \S~\ref{subsection:3-26}\\
\hline
\textnumero {3.28} & $\mathbb{P}^1\times\mathbb{F}_1$, where $\mathbb{F}_1$ is the 1st Hirzebruch~surface & \S~\ref{subsection:3-28}\\
\hline
 \shortstack{\textnumero {3.29}\\{}\\{}} & \shortstack{{ }\\blow-up of $V_7$ in a~line in \\ the~exceptional divisor of the~blow-up $V_7\to\mathbb{P}^3$} & \shortstack{\S~\ref{subsection:3-29}\\{}\\{}} \\
\hline
  \textnumero {3.30} & blow-up of $\mathbb{P}^1$-bundle $V_7\to\mathbb{P}^2$ in a~fibre & \S~\ref{subsection:3-30}\\
\hline
  \textnumero {3.31} & blow-up of the~quadric cone in its vertex & \S~\ref{subsection:3-31}\\
\hline
\textnumero {4.8}& blow-up of $(\mathbb{P}^1)^3$ along curve of degree $(0,1,1)$ & \S~\ref{subsection:4-8}\\
\hline
\shortstack{ \textnumero 4.9\\{ }\\{ }\\{ }\\{ }} & \shortstack{\\blow-up of $V_7$ along proper transform  of two\\ skew lines in $\mathbb{P}^3$, one through blown-up point} & \shortstack{\S~\ref{subsection:4-9}\\{ }\\{ }}\\
\hline
\shortstack{\textnumero {4.10}\\{}\\{}}  & \shortstack{\\{ }$\mathbb{P}^1\times S_7$, where $S_7$ is the smooth \\ del Pezzo surface of degree $7$} & \shortstack{\S~\ref{subsection:4-10}\\{}\\{}} \\
\hline
 \shortstack{\textnumero {4.11}\\{}\\{}} & \shortstack{{ }\\ blow-up of $\mathbb{P}^1\times\mathbb{F}_1$ in a~$(-1)$-curve of a\\ fibre of the~projection $\mathbb{P}^1\times\mathbb{F}_1\to\mathbb{P}^1$}  & \shortstack{\S~\ref{subsection:4-11}\\{}\\{}} \\
\hline
\shortstack{\textnumero {4.12}\\{}\\{}} & \shortstack{{ }\\ blow-up of the~smooth Fano \textnumero 2.33 in two\\ curves contracted by morphism to~$\mathbb{P}^3$} & \shortstack{\S~\ref{subsection:4-12}\\{}\\{}}\\
\hline
\end{tabular}
\end{center}
\end{table}

Throughout this paper, all varieties are assumed to be projective, normal, and defined over $\mathbb{C}$.

\medskip
\noindent
\textbf{Acknowledgements.}
Hamid Abban is supported by the Royal Society International Collaboration Award ICA$\backslash$1$\backslash$231019.
Paolo Cascini and Ivan Cheltsov are supported by the Simons Collaboration Grant \textit{Moduli of varieties}.
All intermediate symbolic computations were assisted by Maple. We would like to thank Thibaut Delcroix, Chi Li, Cristiano Spotti, Hendrik S\"u\ss,\ Xiaowei Wang, Chenyang Xu and Kewei Zhang for very helpful discussions.

\section{Degenerations and Futaki invariant}
\label{section:degeneration}
In this section we first use a direct argument based on the definition of the Donaldson--Futaki invariant and test configurations
to prove that $(X,L)$ is never K-polystable for smooth Fano threefolds $X$ in Families \textnumero 2.21 and \textnumero 3.13 that have non-reductive automorphism groups,
where $L$ is an ample $\mathbb{Q}$-divisor on $X$. Since K-stability of $(X,L)$ depends only on the class $\mathbb{Q}_{>0}[L]$, we will often assume that $L$ is an ample line bundle on $X$.

\begin{definition}
\label{test-configuration}
Let $\mathcal{X}$ be a normal variety equipped with a line
bundle $\mathcal{L}$, admitting a $\mathbb{C}^*$-action, and let $\pi\colon\mathcal{X}\to\mathbb{A}^1$ be a flat projective morphism
with $\mathcal{L}$ relatively ample, where $\pi$ is equivariant with  respect to the usual $\mathbb{C}^*$-action on $\mathbb{A}^1$.
Then $\mathcal{(X,L)}$ together with this equivariant information is called a {\it test configuration} when
$$
(\mathcal{X}_t,\mathcal{L}_t)\cong (X,L)
$$
for all  $t\neq 0$, where $L$ is an ample line bundle on $X$,  $\mathcal{X}_t=\pi^{-1}(t)$ and $\mathcal{L}_t=\mathcal{L}\vert_{\mathcal{X}_{t}}$.

For any $k\in\mathbb{Z}_{>0}$ there is an induced $\mathbb{C}^*$-action on $H^0(\mathcal{X}_0,\mathcal{L}^k_0)$ with total weight $w(k)$, for large enough $k$, a
polynomial of degree $n+1$ of the form
$$
w(k)=b_0k^{n+1}+b_1k^n+\cdots.
$$
On the other hand, the corresponding Hilbert polynomial is of the form
$$
h(k)= a_0k^n + a_1k^{n-1} +\cdots,
$$
where $h(k)=\dim H^0(\mathcal{X}_0,\mathcal{L}^k_0)$. The Donaldson--Futaki invariant of $\mathcal{(X,L)}$ is defined to be
$$
\mathrm{DF}\mathcal{(X,L)} = \frac{b_0a_1-b_1a_0}{a_0}.
$$
With this setup, $(X,L)$ is said to be {\it K-semistable} if for all test configurations $\mathcal{(X,L)}$, we have
$$
\mathrm{DF}\mathcal{(X,L)}\geqslant 0,
$$
and it is called {\it K-polystable} if it is K-semistable, and the equality $\mathrm{DF}\mathcal{(X,L)}=0$ holds
if and only if $(\mathcal{X}_0,\mathcal{L}_0)\cong (X,L)$ (in this case we say that $\mathcal{(X,L)}$ is a product test configuration).
\end{definition}

We now recall the Atiyah-Bott formula, which will be used to compute the Donaldson--Futaki invariant for
Family \textnumero 2.26 and Family \textnumero 3.16.

Let $X$ be a smooth projective variety of dimension $n$, let $g$ be an automorphism of $X$, and let $E$ be a $g$-equivariant vector bundle on $X$.
Assume that the fixed locus $X^g$ of $g$ is a finite set and that the graph $\Gamma_g$ of $g$ in $X\times X$ intersects the diagonal transversally.  Then, the  Atiyah-Bott formula (e.g. see \cite[Appendix to Exp. III, Cor. 6.12]{SGA5}) gives
$$
\sum_{i=0}^{n} (-1)^i \operatorname{Tr}\bigl(g \mid H^i(X,E)\bigr)
=
\sum_{P \in X^g}
\frac{\operatorname{Tr}(g \mid E_P)}{\det(1-g \mid T_P(X))}.
$$

Assume now that $X$ is equipped with an algebraic $\C^*$-action.
Choose $\lambda\in \C^*$ such that $X^\lambda$ is finite and $\det(1-\lambda\mid T_P(X))\neq 0$ for every $P\in X^\lambda$.
Let $g:=\lambda$.
Then $X^g$ is finite and $\Gamma_g$ meets the diagonal in $X\times X$ transversally.
Suppose also that  $E=L^k$, where $L$ is a $\mathbb{C}^*$-linearised ample line bundle and $k$ is a sufficiently large positive integer such that $H^i(X,L^k)=0$ for any $i>0$. Thus, if we denote
$$
\chi_k(\lambda)
=
\operatorname{Tr}\bigl(\lambda \mid H^0(X,L^k)\bigr).
$$
then the Atiyah--Bott formula gives
$$
\chi_k(\lambda)
=
\sum_{P\in X^{\mathbb{C}^*}}
\frac{\lambda^{k\mu_L(P)}}{\prod_{j=1}^n \bigl(1-\lambda^{-\alpha_j(P)}\bigr)},
$$
where, for any point $P\in X^{\mathbb{C}^*}$, $\alpha_1(P),\dots,\alpha_n(P)$ are the weights of the $\mathbb{C}^*$-action on
the cotangent space $T^*_P(X)$, and $\mu_L(P)$ is the weight of the induced $\mathbb{C}^*$-action on the fiber $L_P$ of  $L$ at the point~$P$. To compute the Donaldson--Futaki invariant, with $h(k)$ and $w(k)$ as above, write $\lambda=e^\varepsilon$ so that
$$
\chi_k(e^\varepsilon)
=
h(k)+\varepsilon\,w(k)+O(\varepsilon^2).
$$

Assume now that $X$ is a threefold. Then, we have
$$
h(k)=a_0k^3+a_1k^2+O(k)
\qquad\text{and}\qquad
w(k)=b_0k^4+b_1k^3+O(k^2).
$$
By expanding the  formula at $\lambda=e^\varepsilon$, we obtain:
\begin{equation}
\label{equation:ab}
b_0=\sum_{P\in X^{\C^*}}\frac{\mu^4_L(P)}{24\,\alpha_{1}(P)\alpha_{2}(P)\alpha_{3}(P)}
\quad
\text{and}
\quad
b_1=\sum_{P\in X^{\C^*}}\frac{\mu^3_L(P)(\alpha_{1}(P)+\alpha_{2}(P)+\alpha_{3}(P))}{12\,\alpha_{1}(P)\alpha_{2}(P)\alpha_{3}(P)}.
\end{equation}

In the following two families, we exhibit degenerations of $(X,L)$ to non-isomorphic models with vanishing Donaldson--Futaki invariants, hence showing that $(X,L)$ is not K-polystable.

\subsection{Family \textnumero 2.21}
\label{subsection:2-21}

Smooth Fano threefolds \textnumero 2.21 are blow-ups of the~smooth quadric threefold along a~twisted quartic curve.
It follows from \cite{CheltsovShramovPrzyjalkowski} that there exists a unique such threefold with non-reductive automorphism group,
and the connected component of its automorphism group is isomorphic to the additive group $\mathbb{G}_a$.
The description of this very special smooth Fano threefold is given in \cite{Book,Malbon}.
Namely, we first fix the~twisted quartic curve $C\subset\mathbb{P}^4$ given parametrically by
$$
[u:v]\mapsto[u^4:u^3v:u^2v^2:uv^3:v^4],
$$
and let $Q_{\epsilon}$ be the~quadric threefold in $\mathbb{P}^4$ that is given by
$$
\epsilon(t^2-zw)+3z^2-4yt+xw=0,
$$
where $\epsilon\in\mathbb{C}$.
Then $Q_{\epsilon}$ is smooth, and $Q_{\epsilon}$ contains the curve $C$.
Let $\pi\colon X_{\epsilon}\to Q_{\epsilon}$ be the blow-up of the curve $C$.
Then $X_{\epsilon}$ is a smooth Fano threefold in Family \textnumero 2.21, and it follows from \cite{Malbon} that
$$
\mathrm{Aut}(X_\epsilon)\simeq\mathrm{Aut}^0(X_\epsilon)\times\big(\mathbb{Z}/2\mathbb{Z}\big)\simeq\mathrm{Aut}(Q_{\epsilon},C)\times\big(\mathbb{Z}/2\mathbb{Z}\big).
$$
Moreover, if $\epsilon \ne 0$, then $X_{\epsilon}\simeq X_{1}$, and it follows from \cite{Malbon} that
$
\mathrm{Aut}^0(X_1)\simeq\mathrm{Aut}(Q_1,C)\simeq\mathbb{G}_a
$.
On the other hand, we have $\mathrm{Aut}^0(X_0)\simeq\mathrm{Aut}(Q_0,C)\simeq\mathrm{PGL}_{2}(\mathbb{C})$.

\begin{lemma}
\label{lemma:2-21}
Let $L$ be any ample $\mathbb{Q}$-divisor on $X_1$. Then $(X_1,L)$ is not K-polystable.
\end{lemma}

\begin{proof}
Consider the special test configuration $f\colon \mathcal{X}\to\mathbb{A}^1_{\epsilon}$ whose fibers are $f^{-1}(\epsilon)=X_{\epsilon}$. Then its Donaldson--Futaki invariant coincides (up to a positive multiple) with the Futaki invariant of the $\mathbb{G}_m$-action on the central fiber $X_0$ induced by the corresponding $\mathbb{C}^\ast$-action on $\mathcal{X}$, see \cite{Donaldson2002,Legendre}. On the other hand, it follows from \cite{SektnanTipler} that the Futaki invariant vanishes identically on $X_0$. Thus, by definition, the pair $(X_1,L)$ is not K-polystable.
\end{proof}

\subsection{Family \textnumero 3.13}
\label{subsection:3-13}

Let $X$ be the complete intersection in~$\mathbb{P}^2_{x_1,x_2,x_3}\times\mathbb{P}^2_{y_1,y_2,y_3}\times\mathbb{P}^2_{z_1,z_2,z_3}$ that is given by
$$
\left\{\aligned
&x_1y_1+x_2y_2+x_3y_3=0,\\
&y_1z_1+y_2z_2+y_3z_3=0,\\
&x_1z_2+x_2z_1+x_2z_3-x_3z_2-2x_3z_3=0.\\
\endaligned
\right.
$$
Then $X$ is a smooth Fano threefold in Family~\textnumero 3.13, and it is the only smooth member of this deformation family that is not K-polystable \cite{Book}.
For every $a\in\mathbb{C}$, let $\phi_a\in\mathrm{Aut}(X)$ be given by
$$
\left[
  \begin{array}{c}
    x_1 \\
    x_2\\
    x_3 \\
  \end{array}
\right]\mapsto
A\left[
  \begin{array}{c}
    x_1 \\
    x_2\\
    x_3 \\
  \end{array}
\right],\ \left[
  \begin{array}{c}
    y_1 \\
    y_2\\
    y_3 \\
  \end{array}
\right]\mapsto
\big(A^{-1}\big)^T\left[
  \begin{array}{c}
    y_1 \\
    y_2\\
    y_3 \\
  \end{array}
\right],
\left[
  \begin{array}{c}
    z_1 \\
    z_2\\
    z_3 \\
  \end{array}
\right]\mapsto
A\left[
  \begin{array}{c}
    z_1 \\
    z_2\\
    z_3 \\
  \end{array}
\right],
\text{ where }
A=\left(
  \begin{array}{ccc}
    1 & a^2 & 2a \\
    0 & 1 & 0 \\
    0 & a~& 1 \\
  \end{array}
\right).
$$
Then the transformations $\phi_a$ generate a subgroup of $\mathrm{Aut}(X)$ that is isomorphic to $\mathbb{G}_{a}$, and one can show that $\mathrm{Aut}^0(X)\simeq\mathbb{G}_{a}$.

\begin{lemma}
\label{lemma:3-13}
Let $L$ be any ample $\mathbb{Q}$-divisor on $X$. Then $(X,L)$ is not K-polystable.
\end{lemma}

\begin{proof}
For $\epsilon\in\mathbb{C}$, let $X_{\epsilon}$ be given~by
$$
\left\{\aligned
&x_1y_1+x_2y_2+x_3y_3=0,\\
&y_1z_1+y_2z_2+y_3z_3=0,\\
&x_1z_2+x_2z_1-2x_3z_3+\epsilon(x_2z_3-x_3z_2)=0.
\endaligned
\right.
$$
Then we have $X_{\epsilon}\simeq X$ for every $\epsilon\ne 0$.
On the other hand, $X_0$ is the~unique smooth Fano threefold in Family \textnumero 3.13 that admits an~effective $\mathrm{PGL}_2(\mathbb{C})$-action \cite{Book}.
In fact, $\mathrm{Aut}(X_0)\simeq\mathrm{PGL}_2(\mathbb{C})\times\mathfrak{S}_3$, where $\mathfrak{S}_3$ is the symmetric group of order $6$.
As in the proof of Lemma~\ref{lemma:2-21}, consider the special test configuration $f\colon \mathcal{X}\to\mathbb{A}^1_{\epsilon}$ whose fibres are $f^{-1}(\epsilon)=X_{\epsilon}$.
Since $\mathrm{Aut}^0(X_0)\simeq\mathrm{PGL}_2(\mathbb{C})$, it follows from \cite{Futaki,Mabuchi,PhongSturm} that the Futaki character of $X_0$ vanishes identically on the ample cone of $X_0$. Thus, arguing as in the proof of Lemma~\ref{lemma:2-21}, we see that $(X,L)$ is not K-polystable.
\end{proof}

Next, we treat Families \textnumero 2.26 and \textnumero 3.16 using a different approach, based on $\mathbb{C}^*$-actions and explicit computations of the corresponding Futaki invariant using the Atiyah--Bott formula. 

\subsection{Family \textnumero 2.26}
\label{subsection:2-26}
Let $[x_0:x_1:x_2:x_3:x_4]$ be homogeneous coordinates on $\mathbb{P}^4$, let
$$
Q=\{x_0x_4+x_1x_3+x_2^2=0\}\subset \PP^4,
$$
and let $C_3$ be the parametrically given twisted cubic curve
$$
[s:t]\mapsto [s^3:s^2t:st^2:-t^3:0].
$$
Then the smooth threefold $X$ concerned in this family
is obtained by $\pi\colon X\to Q$, the blow-up of the smooth quadric threefold $Q$ along the curve $C_3$.
Let $E$ be the~$\pi$-exceptional divisor, let $H=\pi^*(\mathcal{O}_{Q}(1))$, and let $L$ be an ample divisor on $X$. Then
$$
L\sim aH+b(2H-E)
$$
for some positive integers $a$ and $b$.
Fix the $\C^*$-action on $\PP^4$ that acts diagonally on $\PP^4$ with weights
$$
(w_0,w_1,w_2,w_3,w_4)=(2,1,0,-1,-2),
$$
which is given by
$$
[x_0:x_1:x_2:x_3:x_4]\mapsto [\lambda^{2}x_0:\lambda x_1:x_2:\lambda^{-1}x_3:\lambda^{-2}x_4]
$$
for $\lambda\in\mathbb{C}^*$. Then $Q$ is $\C^*$-invariant, so we may also consider this $\mathbb{C}^*$ as a subgroup in $\mathrm{Aut}(Q)$.
Moreover, the twisted cubic curve $C_3$ is also $\C^*$-invariant, so the described $\C^*$-action lifts to $X$,
and we may consider $H$, $E$ and $L$ as $\mathbb{C}^*$-linearized line bundles.

\begin{lemma}
\label{lemma:2-26-DF}
For the product test configuration that is induced by the described $\C^*$-action, the~Donaldson--Futaki invariant of $(X,L)$ is
\begin{equation}
\label{equation:2-26-DF}
-\frac{ab\bigl(6a^4+35a^3b+80a^2b^2+75ab^3+25b^4\bigr)}{4(a+b)(2a^2+10ab+5b^2)}.
\end{equation}
\end{lemma}

\begin{proof}
Fix sufficiently large integer $k\gg 0$. Let $h(k)=\dim(H^0(X,L^k))$, and let $w(k)$ be the total weight of our $\C^*$-action on $H^0(X,L^k)$.
Then
\begin{align*}
h(k)&=a_0k^3+a_1k^2+O(k),\\
w(k)&=b_0k^4+b_1k^3+O(k^2),
\end{align*}
and the Donaldson--Futaki invariant is $\frac{b_0a_1-b_1a_0}{a_0}$. Note that up to a non-zero scalar multiple, the Donaldson--Futaki invariant coincides with the usual Futaki invariant of the vector field induced by the $\mathbb{C}^*$-action, which can be computed using \cite[Corollary~2.13]{Maschler}, \cite[\S~3]{Legendre} or \cite[Theorem~1.4]{Tian} similar to what is done in \cite[Example 4.9]{Stoppa2024}. However, we can compute it directly, as by the asymptotic Riemann--Roch, we have
$$
a_0=\frac{L^3}{6}=\frac{(a+b)(2a^2+10ab+5b^2)}{6}
\quad
\text{and}
\quad
a_1=\frac{(-K_X\cdot L^2)}{4}=\frac{3a^2+9ab+5b^2}{2}.
$$
Thus, to obtain the required formula \eqref{equation:2-26-DF}, we  use the formulas \eqref{equation:ab} to compute $b_0$ and $b_1$ explicitly as functions of $a$ and $b$.

First, we describe the set of $\mathbb{C}^*$-fixed points in $X$, which we denote by
$X^{\C^*}$. Our $\mathbb{C}^*$-action fixes exactly $4$ points in $Q$. These points are
$$
Q_0=[1:0:0:0:0],\,
Q_1=[0:1:0:0:0],\,
Q_3=[0:0:0:1:0],\,
Q_4=[0:0:0:0:1].
$$
The twisted cubic $C_3$ contains $Q_0$ and $Q_3$, and it does not contain $Q_1$ and $Q_4$.
We denote by $P_1$ and $P_4$ the preimages on $X$ of the points $Q_1$ and $Q_4$, respectively.
Let $\ell_0$ and $\ell_3$ be the fibres of the natural projection $E\to C_3$ over the points $Q_0$ and $Q_3$, respectively.
Then $\ell_0$ contains two $\mathbb{C}^*$-fixed points, which we denote by $P_0$ and $P_0^\prime$.
Similarly, the curve $\ell_3$ contains two $\mathbb{C}^*$-fixed points, which we denote by $P_3$ and $P_3^\prime$.
Then $X^{\C^*}=\{P_0,P_0^\prime,P_1,P_3,P_3^\prime,P_4\}$.

For every $P\in X^{\C^*}$, the group $\mathbb{C}^*$ faithfully acts on the cotangent space $T_P^*(X)$.
Possibly after swapping $P_0\leftrightarrow P_0^\prime$ and $P_3\leftrightarrow P_3^\prime$,
the weights of the $\mathbb{C}^*$-actions are computed as follows:
\begin{align*}
\big(\alpha_1(P_0),\alpha_2(P_0),\alpha_3(P_0)\big)&=(-1,-2,-1),&\quad
\big(\alpha_1(P_0^\prime),\alpha_2(P_0^\prime),\alpha_3(P_0^\prime)\big)&=(-1,-3,1),\\
\big(\alpha_1(P_1),\alpha_2(P_1),\alpha_3(P_1)\big)&=(1,-1,-3),&\quad
\big(\alpha_1(P_3),\alpha_2(P_3),\alpha_3(P_3)\big)&=(1,3,-4),\\
\big(\alpha_1(P_3^\prime),\alpha_2(P_3^\prime),\alpha_3(P_3^\prime)\big)&=(1,-1,4),&\quad
\big(\alpha_1(P_4),\alpha_2(P_4),\alpha_3(P_4)\big)&=(3,2,1).
\end{align*}
Similarly, we compute the fibre weights for the line bundles $H$ and $E$ as follows:
$$
\mu_H(P_0)=\mu_H(P_0^\prime)=-2,\quad \mu_H(P_1)=-1,\quad \mu_H(P_3)=\mu_H(P_3^\prime)=1,\quad \mu_H(P_4)=2,
$$
$$
\mu_E(P_0)=-2, \quad \mu_E(P_0^\prime)=-3,\quad \mu_E(P_1)=0,\quad \mu_E(P_3)=3,\quad  \mu_E(P_3^\prime)=-1, \quad \mu_E(P_4)=0.
$$
Thus, the fibre weight for the line bundle $L$ is
\begin{align*}
\mu_L(P_0)&=-2a-2b,\quad
\mu_L(P_0^\prime)=-2a-b,\quad
\mu_L(P_1)=-a-2b,\\
\mu_L(P_3)&=a-b,\quad
\mu_L(P_3^\prime)=a+3b,\quad
\mu_L(P_4)=2(a+2b).
\end{align*}
Now, using the formulas for $b_0$ and $b_1$ presented in \eqref{equation:ab}, we compute
$$
b_0=\frac{b^2(9a^2+20ab+10b^2)}{12}
\quad
\text{and}
\quad
b_1=\frac{b(3a^2+13ab+10b^2)}{4}.
$$
Finally, using these formulas and formulas for $a_0$ and $a_1$ obtained earlier, we get \eqref{equation:2-26-DF}.
\end{proof}

Therefore, the polarised pair $(X,L)$ in this family is never K-semistable.

\subsection{Family \textnumero 3.16}
\label{subsection:3-16}

Let $O=[1:0:0:0]\in\mathbb{P}^3$, and let $C_3$ be the twisted cubic curve in $\mathbb{P}^3$ that is parametrically given by
$$
[s:t]\mapsto [s^3:s^2t:st^2:t^3].
$$
Then $O\in C_3$.
Let~$\phi\colon V_7\to\mathbb{P}^3$ be the~blow-up of the point $O$, let $C$ be the~proper transform on $V_7$ of the~curve $C_3$,
and let $\pi\colon X\to V_7$ be the blow-up along $C$.
Then $X$ is the smooth Fano threefold in Family \textnumero 3.16,
$\mathrm{Aut}^0(X)\simeq\mathbb{C}_+\rtimes\mathbb{C}^\ast$,
and we have the following commutative diagram:
$$
\xymatrix{
\mathbb{P}^2& &W\ar@{->}[ll]_{\mathrm{pr}_1}\ar@{->}[rr]^{\mathrm{pr}_2} & &\mathbb{P}^2 \\
& & X\ar@{->}[drr]^{\varphi}\ar@{->}[dll]_{\pi}\ar@{->}[u]_{\psi} & & \\
V_7\ar@{->}[rr]_{\phi}\ar@{->}[uu]_{p_1}& &\mathbb{P}^3  & &\widetilde{\mathbb{P}}^3\ar@{->}[ll]^{\varpi}\ar@{->}[uu]_{p_2}}
$$
where $W$ is a~smooth divisor of degree $(1,1)$ in $\mathbb{P}^2\times\mathbb{P}^2$,
the morphism $\psi$ is a blow-up of a smooth curve of degree $(2,1)$,
$\varpi$ is the~blow-up of the curve $C_3$,
$\varphi$ is the~blow-up of the~fibre of $\varpi$ over $O$,
and $\mathrm{pr}_1$, $\mathrm{pr}_2$, $p_1$, $p_2$ are $\mathbb{P}^1$-bundles.

Let $E$ be the~$\varphi$-exceptional surface,
let $F$ be the~$\pi$-exceptional surface,
let $G$ be the~$\psi$-exceptional surface,
and let $H$ be the strict transform on $X$ of a general plane in $\mathbb{P}^3$.
Then $E\simeq\mathbb{F}_1$, $F\simeq\mathbb{P}^1\times\mathbb{P}^1$, $G\simeq\mathbb{F}_2$, $G\sim 2H-2E-F$,
the nef cone of $X$ is generated by $H$, $H-E$ and $2H-E-F$. Note also that
$-K_X\sim 4H-2E-F$.

Let $L$ be an ample divisor on $X$. Then there are positive integers $a$, $b$, $c$ such that  
$$
L\sim_{\mathbb{Q}} aH+b(H-E)+c(2H-E-F).
$$
As in the previous case, we fix the $\mathbb{C}^*$-action on $\mathbb{P}^3$ that is given by
$$
[x_0:x_1:x_2:x_3]\mapsto [x_0:\lambda x_1:\lambda^2x_2:\lambda^3x_3],
$$
where $\lambda\in\mathbb{C}^*$. Then $O$ is fixed by this $\mathbb{C}^*$-action, and $\mathscr{C}$ is $\mathbb{C}^*$-invariant,
so the $\mathbb{C}^*$-action lifts to the threefolds $V_7$, $\widetilde{\mathbb{P}}^3$, $X$,
and we may consider $H$, $E$ and $L$ as $\mathbb{C}^*$-linearized line bundles.

\begin{lemma}
\label{lemma:3-16-DF}
For the product test configuration induced by the described $\C^*$-action on $X$, the~Donaldson--Futaki invariant of $(X,L)$ is
$$
-\frac{f(a,b,c)}{4(a^{3} + 3 a^{2} b + 6 a^{2} c + 3 a b^{2} + 12 a b c + 3 a c^{2} + 3 b^{2} c + 3 b c^{2})},
$$
where
\begin{multline*}
f(a,b,c)=3a^{4} b^{2} + 6 a^{4} b c + 6 a^{4} c^{2} + 4 a^{3} b^{3} + 25 a^{3} b^{2} c + 67 a^{3} b c^{2} + 48 a^{3} c^{3}+\\
+ 24 a^{2} b^{3} c + 129 a^{2} b^{2} c^{2} + 204 a^{2} b c^{3} + 90 a^{2} c^{4}+ 12 a b^{4} c + 102 a b^{3} c^{2} + \\
+243 a b^{2} c^{3} + 201 a b c^{4} + 36 a c^{5}+ 12 b^{4} c^{2} + 54 b^{3} c^{3} + 78 b^{2} c^{4} + 36 b c^{5}.
\end{multline*}
\end{lemma}

\begin{proof}
The proof is very similar to  the proof of Lemma~\ref{lemma:2-26-DF}. For convenience of the reader, we present some details.
The Donaldson--Futaki invariant of $(X,L)$ is $\frac{b_0a_1-b_1a_0}{a_0}$, where
$$
a_0=\frac{L^3}{6}=\frac{a^3+3a^2b+6a^2c+3ab^2+12abc+3ac^2+3b^2c+3bc^2}{6},
$$
$$
a_1=\frac{(-K_X\cdot L^2)}{4}=\frac{2(2a^2+4ab+5ac+b^2+4bc+c^2)}{4},
$$
and $b_0$ and $b_1$ can be computed using  \eqref{equation:ab} as follows.

Let $X^{\C^*}$ be the set of $\mathbb{C}^*$-fixed points in $X$.
Then $X^{\C^*}$ consists of eight points.
To describe them, observe that our $\mathbb{C}^*$-action fixes exactly $4$ points in $\mathbb{P}^3$. These points are
$$
O=[1:0:0:0],\,
Q_1=[0:1:0:0],\,
Q_2=[0:0:1:0],\,
Q_3=[0:0:0:1].
$$
The twisted cubic $C_3$ contains $O$ and $Q_3$, and it does not contain $Q_1$ and $Q_2$.
We denote by $\bar{Q}_1$, $\bar{Q}_2$, $\bar{Q}_3$ the preimages in $V_7$ of the points $Q_1$, $Q_2$, $Q_3$, respectively.
Set $\overline{E}=\pi(E)$. Then $\overline{E}\simeq\mathbb{P}^2$ is $\phi$-exceptional divisor,
and our $\C^*$-action acts on $\overline{E}$ with weights $(1,2,3)$, so it fixes three points,
which we denote by $\bar{Q}_4$, $\bar{Q}_5$, $\bar{Q}_6$.
Note that one of these points is contained in $C$. Without loss of generality, we may assume that $\bar{Q}_4\in C$.
Thus, $C$ contains $\bar{Q}_3$ and $\bar{Q}_4$, and it does not contain  $\bar{Q}_1$, $\bar{Q}_2$, $\bar{Q}_5$, $\bar{Q}_6$.

We denote by $P_1$, $P_2$, $P_5$ and $P_6$ the preimages on $X$ of the points $\bar{Q}_1$, $\bar{Q}_2$, $\bar{Q}_5$, $\bar{Q}_6$, respectively.
Let $\ell_3$ and $\ell_4$ be the fibres of the natural projection $F\to C$ over the points $\bar{Q}_3$ and $\bar{Q}_4$, respectively.
Then $\ell_3$ contains two $\mathbb{C}^*$-fixed points, which we denote by $P_3$ and $P_3^\prime$.
Similarly, the curve $\ell_4$ contains two $\mathbb{C}^*$-fixed points, which we denote by $P_4$ and $P_4^\prime$.
Then
$$
X^{\C^*}=\big\{P_1,P_2,P_3,P_3^\prime,P_4,P_4^\prime,P_5,P_6\big\}.
$$
By construction, $P_1$ and $P_2$ are not contained in $E\cup F$, while  $P_3$ and $P_3^\prime$ are contained in $F$,
but $P_3$ and $P_3^\prime$ are not contained in $E$.
Similarly, points $P_4$ and $P_4^\prime$ are contained in $\ell_4=E\cap F$, while points $P_5$ and $P_6$ are contained in $E$,
but they are not contained in $F$.

For every $P\in X^{\C^*}$, the group $\mathbb{C}^*$ faithfully acts on the cotangent space $T_P^*(X)$,
and we denote by $\alpha_1(P)$, $\alpha_2(P)$, $\alpha_3(P)$ the weights of this action.
Possibly after swapping $P_3\leftrightarrow P_3^\prime$, $P_4\leftrightarrow P_4^\prime$, $P_5\leftrightarrow P_6$, 
the weights of these eight $\mathbb{C}^*$-actions are computed as follows:
\begin{align*}
\big(\alpha_1(P_1),\alpha_2(P_1),\alpha_3(P_1)\big)&=(-1,1,2),&\quad
\big(\alpha_1(P_2),\alpha_2(P_2),\alpha_3(P_2)\big)&=(-2,-1,1),&\quad\\
\big(\alpha_1(P_3),\alpha_2(P_3),\alpha_3(P_3)\big)&=(-2,-1,-1),&\quad
\big(\alpha_1(P_3^\prime),\alpha_2(P_3^\prime),\alpha_3(P_3^\prime)\big)&=(1,-3,-1),\\
\big(\alpha_1(P_4),\alpha_2(P_4),\alpha_3(P_4)\big)&=(-1,2,1),&\quad
\big(\alpha_1(P_4^\prime),\alpha_2(P_4^\prime),\alpha_3(P_4^\prime)\big)&=(1,1,1),\\
\big(\alpha_1(P_5),\alpha_2(P_5),\alpha_3(P_5)\big)&=(-1,2,1),&\quad
\big(\alpha_1(P_6),\alpha_2(P_6),\alpha_3(P_6)\big)&=(-2,-1,3),&\quad
\end{align*}
Similarly, we compute the fibre weights for  $H$, $E$ and $F$ as follows:
\begin{multline*}
\quad \quad \quad \quad \quad \mu_H(P_1)=1, \mu_H(P_2)=2, \mu_H(P_3)=3, \mu_H(P_3^\prime)=3, \\
\mu_H(P_4)=0, \mu_H(P_4^\prime)=0, \mu_H(P_5)=0, \mu_H(P_6)=0,\quad \\
\mu_E(P_1)=0, \mu_E(P_2)=0, \mu_E(P_3)=0, \mu_E(P_3^\prime)=0, \\
\mu_E(P_4)=-1, \mu_E(P_4^\prime)=-1, \mu_E(P_5)=-2, \mu_E(P_6)=-3, \\
\mu_F(P_1)=0, \mu_F(P_2)=0, \mu_F(P_3)=2, \mu_F(P_3^\prime)=3, \\
\mu_F(P_4)=-2, \mu_F(P_4^\prime)=-1,  \mu_F(P_5)=0, \mu_F(P_6)=0.\quad \quad \quad \quad \quad
\end{multline*}
Thus, the fibre weight for the line bundle $L$ is
\begin{multline*}
\mu_L(P_1)=a+b+2c,\quad
\mu_L(P_2)=2a+2b+4c,\quad
\mu_L(P_3)=3a+3b+4c,\quad
\mu_L(P_3^\prime)=3a+3b+3c,\quad\\
\mu_L(P_4)=b+3c,\quad
\mu_L(P_4^\prime)=b+2c,\quad
\mu_L(P_5)=2b+2c,\quad
\mu_L(P_6)=3b+3c,\quad
\end{multline*}

Now, using \eqref{equation:ab}, we compute $b_0$ and $b_1$ in $\frac{b_0a_1-b_1a_0}{a_0}$, which results
in the required formula for the Donaldson--Futaki invariant.
\end{proof}

Hence, the pair $(X,L)$ in this family is K-unstable for any choice of ample divisor $L$.

\section{Destabilisation via the $\beta$-invariant}
\label{section:beta}
In this section, we apply the $\beta$-invariant method mentioned in the proof of the Main Theorem.
In each case, we identify a divisor $F$ whose geometry reflects a natural extremal contraction and show that $\beta_L(F) < 0$.
Before we proceed with each case, we introduce the necessary concepts and develop some tools to simplify calculations.

Let $X$ be a normal projective variety of dimension~$n$ that has at most Kawamata log terminal singularities, and let $L$ be a nef and big $\mathbb{Q}$-divisor on $X$. Set
$$
\mu=\frac{(-K_X)\cdot L^{n-1}}{L^n}.
$$
Then, following \cite{DervanLegendre}, for every prime divisor $\mathscr{F}$ over $X$, we let
\begin{equation}
\label{equation:beta}\tag{$\bigstar$}
\beta_L(\mathscr{F})=A_X(\mathscr{F})+\frac{n\mu}{L^n}\int\limits_{0}^{\infty}\mathrm{vol}\big(L-u\mathscr{F}\big)du+
\frac{1}{L^n}\int\limits_{0}^{\infty}\frac{\partial \mathrm{vol}\big(L+tK_X-u\mathscr{F}\big)}{\partial t}\Big\vert_{t=0}du,
\end{equation}
where $A_X(\mathscr{F})$ denotes the log discrepancy with respect to $\mathscr{F}$.
Note that $\beta_L(\mathscr{F})$ is invariant up to rescaling $L$ by a positive rational number. Moreover, if $X$ is a Fano variety and $L=-K_X$, then $\beta_L(\mathscr{F})$ is the usual $\beta$-invariant \cite{DervanLegendre,Fujita2019}. Furthermore, \cite[Theorem~1.1]{DervanLegendre} gives

\begin{theorem}
\label{theorem:DervanLegendre}
Suppose that $L$ is ample. If $\beta_L(\mathscr{F})<0$ for some prime divisor $\mathscr{F}$ over $X$, then $(X,L)$ is K-unstable.
\end{theorem}

The following result is fruitful in simplifying the computation of $\beta_L(\mathscr{F})$ in some cases.

\begin{proposition}
\label{proposition:Paolo-1}
Suppose that $X$ is a Fano variety, and $L=-K_X+bF$ for a $\mathbb Q$-Cartier prime divisor $F$ on $X$, where $b\in \mathbb Q$ such that $L$ is nef and big. Set
$$
\varphi(b)=\frac{-(L+nbF)\cdot L^{n-1}}{(L^n)^2}.
$$
Then
$$
\beta_L(F)=1+b+\varphi(b)\cdot \int\limits_{-b}^{\infty}{\rm vol}(-K_X-uF)du.
$$
\end{proposition}

\begin{proof} Let $f(u):={\rm vol}(-K_X-uF)$. Then for any $t\in [0,1)$, we have
$$\begin{aligned}
{\rm vol}(L+tK_X-uF)&={\rm vol}(-(1-t)K_X+(b-u)F)\\
&= (1-t)^n {\rm vol}\left(-K_X+\frac{b-u}{1-t}F\right)=(1-t)^nf\left(\frac{u-b}{1-t}\right).
\end{aligned}$$
Thus,
$$\frac{\partial \mathrm{vol}\big(L+tK_X-uF\big)}{\partial t}\Big\vert_{t=0}=
-nf(u-b)+f'(u-b)\cdot (u-b).$$
Therefore, \eqref{equation:beta} becomes
$$\begin{aligned}
\beta_L(F)&=1+\frac{n(\mu-1)}{L^n}\int\limits_0^{\infty}f(u-b)du
+\frac 1 {L^n}\int\limits_0^\infty f'(u-b)(u-b)du\\
&=1+\frac{n(\mu-1)}{L^n}\int\limits_{-b}^\infty f(u)du +
\frac 1 {L^n}\int\limits_{-b}^\infty f'(u) u du.
\end{aligned}
$$
Since $f(u)=0$ for any $u\gg 0$, we have
$$\int\limits_{-b}^{\infty} f'(u)udu = -f(-b)(-b) - \int\limits_{-b}^\infty f(u)du=b L^n -\int\limits_{-b}^{\infty}f(u)du.$$
Thus,
$$
\beta_L(F)=1+ \frac {n(\mu-1)-1}{L^n}\int\limits_{-b}^\infty f(u)du+b=1+b+\varphi(b)\int\limits_{-b}^\infty f(u)du,
$$
and our claim follows.
\end{proof}

\subsection{Family \textnumero 2.28}\label{subsection:2-28}

Let $X$ be the unique smooth Fano threefold in Family \textnumero 2.28.
Then there exists a birational morphism $f\colon X\to\mathbb{P}^3$ such that $f$ is a blow-up of a smooth plane cubic curve $C\subset\mathbb{P}^3$.
We proceed in this case by applying Proposition~\ref{proposition:Paolo-1}.

\begin{lemma}
\label{example:2-28}
Let $L$ be an ample $\mathbb Q$-divisor on $X$.
Then $(X,L)$ is K-unstable for any polarisation.
\end{lemma}

\begin{proof}
Let $\Pi$ be the plane in $\mathbb{P}^3$ that contains $C$,
let $S$ be its strict transform on $X$, and  let $E$ be the $f$-exceptional divisor.
Then we have the following Sarkisov link:
$$
\xymatrix{
\mathbb{P}^3&X\ar[l]_{f}\ar[r]^{g}&V_3}
$$
where $V_3$ is a cubic threefold in $\mathbb{P}(1,1,1,1,2)$, and $g$ is the contraction of $S$ to its singular point.
Moreover, after appropriate scaling, we have $L\sim_{\mathbb{Q}} -K_X+bS$ for some $b\in\mathbb{Q}$ with $-1<b<\frac{1}{2}$.
By Proposition~\ref{proposition:Paolo-1}, we have
$$
\beta_{L}(S)=1+b+\varphi(b)\int\limits_{-b}^\infty\mathrm{vol}(-K_X-uS)du,
$$
where
$$
\varphi(b)=\frac{-(L+3bS)\cdot L^2)}{(L^3)^2}=\frac{-2(b+1)(8b^2-17b+20)}{(4b^3-6b^2+3b+40)^2}.
$$
Now, arguing as in the proof of \cite[Lemma~3.20]{Book}, we get
$$
\mathrm{vol}(-K_X-uS)=\left\{\aligned
&40-4u^3-6u^2-3u\ \text{if $-\frac{1}{2}\leqslant u\leqslant 1$}, \\
&(4-u)^3\ \text{if $1\leqslant u\leqslant 4$}\\
\endaligned
\right.
$$
and integration results in
$$
\beta_{L}(S)=\frac{9(b+1)(4b^5-16b^4-50b^3-14b^2+119b-140)}{2(4b^3-6b^2+3b+40)^2}.
$$
Then $\beta_L(S)<0$ and, in particular, $(X,L)$ is K-unstable.
\end{proof}

\subsection{Family \textnumero 3.18}
\label{subsection:3-18}

Let $\ell$ be a line in $\mathbb{P}^3$, let $C$ be a smooth conic in $\mathbb{P}^3$ such that $C\cap \ell =\varnothing$,
and let $\pi\colon X\to \mathbb{P}^3$ be the blow-up of these curves.
Then $X$ is the unique smooth Fano threefold in Family \textnumero 3.18. Let $\Pi$ be the~strict transform on $X$ of the~plane in $\mathbb{P}^3$ that contains $C$.
Then we have the following commutative diagram:
$$
\xymatrix{
Q& & \widetilde{Q}\ar@{->}[ll]\ar@{->}[rr] & &\mathbb{P}^1 \\
& & X\ar@{->}[d]^{\pi}\ar@{->}[drr]^{\phi}\ar@{->}[dll]_{\theta}\ar@{->}[u]_{\psi} & & \\
V\ar@{->}[rr]_{\varphi}\ar@{->}[uu]^{\eta}& &\mathbb{P}^3 & &Y\ar@{->}[ll]^{\vartheta}\ar@{->}[uu]}
$$
where $Q$ is a~smooth quadric threefold,
$\widetilde{Q}\to Q$ is the~blow-up of a~conic, which is the~strict transform of the~line $\ell$,
$\psi$ is the~contraction of $\Pi$ to a point in $\widetilde{Q}$,
$\eta$ is a~blow-up of a~point in $Q$,
$\vartheta$ is the~blow-up of the~line~$\ell$,
$\varphi$ is the~blow-up of the~conic $C$,
$\theta$ and $\phi$ are the~blow-ups of the~strict transforms of the~curves $\ell$ and $C$, respectively,
\mbox{$Y\to\mathbb{P}^1$} is a~$\mathbb{P}^2$-bundle,
and $\widetilde{Q}\to\mathbb{P}^1$ is a~fibration into quadric surfaces.

Let $F$ and $E$ be the~$\pi$-exceptional surfaces that are mapped to $C$ and $\ell$, respectively, and let $H$ be the~proper transform on $X$ of a~general plane in $\mathbb{P}^3$.
Then the nef cone is generated by $H$, $H-E$ and $2H-F$. Moreover, $-K_X\sim 4H-E-F$.
Note that non-zero intersections of the divisors $H$, $E$, $F$ can be computed as follows:
$H^3=1,\ H\cdot E^2=-1, H\cdot F^2=-2, E^3=-2, F^3=-6$.

Let $L$ be an ample $\mathbb{Q}$-divisor on $X$. Then
$$
L\sim_{\mathbb{Q}} aH+b(H-E)+c(2H-F)
$$
for some $a,b,c\in\mathbb{Q}_{>0}$. Then $L^3=a^3+3a^2b+6a^2c+12abc+6ac^2+6bc^2+2c^3$
and
$$
\frac{(-K_X)\cdot L^{2}}{L^3}=2\frac{2a^2+3ab+6ac+4bc+3c^2}{a^3+3a^2b+6a^2c+12abc+6ac^2+6bc^2+2c^3}.
$$
Set $g(a,b,c)=2(a^3+3a^2b+6a^2c+12abc+6ac^2+6bc^2+2c^3)^2$.

\begin{lemma}
\label{lemma:3-18}
One has
$$
\beta_{L}(\Pi)=-3a\frac{f_1(a,b,c)}{g(a,b,c)}\ \quad \text{and}\quad \beta_{L}(E)=-\frac{f_2(a,b,c)}{g(a,b,c)},
$$
where
\begin{multline*}
\quad \quad \quad \quad f_1(a,b,c)=-3a^4b+6a^4c+27a^3c^2+24a^2b^2c+24a^2bc^2+\\
+48a^2c^3+72ab^2c^2+60abc^3+24ac^4+72b^2c^3+24bc^4, \quad \quad \quad \quad
\end{multline*}
\begin{multline*}
\quad \quad \quad \quad f_2(a,b,c)=3a^5b-6a^5c-21a^4c^2-24a^3b^2c-28a^3c^3-54a^2b^2c^2+\\
+12a^2bc^3-24a^2c^4-36ab^2c^3+6abc^4-12ac^5-12b^2c^4-2c^6.\quad \quad \quad \quad \quad \quad
\end{multline*}
\end{lemma}

\begin{proof}
First, we observe that $\Pi\cong\mathbb{F}_1$.
Denote by $\mathbf{s}$ the~unique $(-1)$-curve in the~surface $\Pi$, denote by $\mathbf{f}$ and $\ell$ the fibres of the~natural projections~\mbox{$\Pi\to\mathbb{P}^1$} and $F\to C$, respectively.
Then the~curves $\mathbf{s}$, $\mathbf{f}$ and $\ell$ generate the~Mori cone $\overline{\mathrm{NE}}(X)$,
and the~corresponding extremal contractions are $\phi$, $\psi$ and $\theta$, respectively.
Note also that $\Pi\vert_{\Pi}\sim-\mathbf{s}-\mathbf{f}$ and $\Pi\cdot\ell=1$.
Therefore, for $u\in\mathbb{R}_{\geqslant 0}$, we have
\[
(L-u\Pi)\cdot\mathbf{s}=b,\quad
(L-u\Pi)\cdot\mathbf{f}=a+u,\quad
(L-u\Pi)\cdot\ell=c-u.
\]
This shows that $L-u\Pi$ is nef for $u\in[0,c]$. Similarly, if $u\in[c,a+2c]$, then the positive part of the Zariski decomposition of the divisor $L-u\Pi$ is $(a+b+2c-u)H-bE$, and its negative part is $(u-c)F$. For $u>a+2c$, the divisor $L-u\Pi$ is not pseudoeffective. Now, using \eqref{equation:beta}, we obtain the required formula for $\beta_{L}(\Pi)$. The formula for $\beta_{L}(E)$ can be obtained in a similar way.
\end{proof}

\begin{lemma}
\label{lemma:3-18-unstable}
Suppose that $L$ is ample. Then $\beta_{L}(\Pi)<0$ or $\beta_{L}(E)<0$.
\end{lemma}

\begin{proof}
Note that $a$, $b$, $c$ are positive. Suppose that $\beta_{L}(\Pi)\geqslant 0$ and $\beta_{L}(E)\geqslant 0$. Then $f_1(a,b,c)\leqslant 0$ and $f_2(a,b,c)\leqslant 0$. Set
$$
\Delta(a,b,c)=3a^4+12a^3b+10a^3c+9a^2b^2+36a^2bc+18ab^2c+15abc^2-6ac^3-6b^2c^2-c^4.
$$
Then $af_1(a,b,c)+f_2(a,b,c)=2c^2\Delta(a,b,c)$. Thus, we see that $\Delta(a,b,c)\leqslant 0$. If $a\geqslant c$, then
\begin{multline*}
0\geqslant \Delta(a,b,c)=21b^2c^2+36b^2c\epsilon+9b^2\epsilon^2+63bc^3+123bc^2\epsilon+72bc\epsilon^2+\\
+12b\epsilon^3+6c^4+36c^3\epsilon+48c^2\epsilon^2+22c\epsilon^3+3\epsilon^4>0,
\end{multline*}
where $\epsilon=a-c$. Similarly, if $c\geqslant a$, then
\begin{multline*}
0\geqslant f_1(a,b,c)=105a^5+105a^4b+300a^4\delta+168a^3b^2+324a^3b\delta+315a^3\delta^2+384a^2b^2\delta+\\
+348a^2b\delta^2+144a^2\delta^3+288ab^2\delta^2+156ab\delta^3+24a\delta^4+72b^2\delta^3+24b\delta^4>0,
\end{multline*}
where $\delta=c-a$. The obtained contradiction completes the proof of the lemma.
\end{proof}

Hence, if $L$ is ample, then $(X,L)$ is K-unstable.

\subsection{Family \textnumero 3.21}
\label{subsection:3-21}

Let $X$ be the unique smooth Fano threefold in Family \textnumero 3.21.
Then it follows from \cite{IsPr99} that there exists a birational morphism $\pi\colon X\to\mathbb{P}^1\times\mathbb{P}^2$ that is a blow-up of a~smooth curve $C$ of degree~$(2,1)$.
Let $E$ be the $\pi$-exceptional surface, let $S$ be the~proper transform on $X$ of the~surface in $\mathbb{P}^1\times\mathbb{P}^2$ of degree $(0,1)$ that passes through $C$, let $H_1=(\mathrm{pr}_1\circ\pi)^*(\mathcal{O}_{\mathbb{P}^1}(1))$ and $H_2=(\mathrm{pr}_2\circ\pi)^*(\mathcal{O}_{\mathbb{P}^2}(1))$, where $\mathrm{pr}_1$ and $\mathrm{pr}_2$ are projections of $\mathbb{P}^1\times\mathbb{P}^2$ to the first and the second factors, respectively. Then $S\sim H_2-E$, and $S\vert_{S}$ is a divisor on $S\simeq\mathbb{P}^1\times\mathbb{P}^1$ of degree $(-1,-1)$. Moreover, it follows from  the~proof of \cite[Lemma~8.22]{CheltsovShramovUMN} that there~is a~commutative diagram
$$
\xymatrix{
\mathbb{P}^1& &\mathbb{P}^1\times\mathbb{P}^2\ar@{->}[ll]\ar@{->}[rr] & &\mathbb{P}^2 \\
& & X\ar@{->}[drr]\ar@{->}[dll]\ar@{->}[u]_{\pi} & & \\
U_1\ar@{->}[rr]\ar@{->}[uu]& &V & &U_2\ar@{->}[ll]\ar@{->}[uu]}
$$
where $X\to U_1$ and $X\to U_2$ are birational contractions of the surface $S$ such that the~induced birational map $U_1\dashrightarrow U_2$ is an~Atiyah flop, the morphism $U_1\to \mathbb{P}^1$ is a~fibration into quadric surfaces, the~morphism $U_2\to \mathbb{P}^2$ is a~$\mathbb{P}^1$-bundle, and $V$ is a~singular Fano threefold with $\mathrm{Pic}(V)=\mathbb{Z}$ and $(-K_V)^3=40$ that has one ordinary double singularity, which can be smoothed to a Fano threefold in Family~\textnumero 1.15. We refer the~reader to the~case~(2.3.2) in~\cite[Theorem~2.3]{TakeuchiNew}.

To describe the nef cone of $X$, observe that the morphisms $X\to\mathbb{P}^1$, $X\to\mathbb{P}^2$ and $X\to V$ in the commutative diagram above are given by the linear systems $|H_1|$, $|H_2|$ and $|H_1+2H_2-E|$, respectively. Thus, the nef cone of $X$ is generated by $H_1$, $H_2$, and $H_1+2H_2-E$. Similarly, to describe the Mori cone of $X$, let $\ell_1$ and $\ell_2$ be the~rulings of the~surface $S\simeq\mathbb{P}^1\times\mathbb{P}^1$  such that the~curves $\pi(\ell_1)$ and $\pi(\ell_2)$ are curves in  $\mathbb{P}^1\times\mathbb{P}^2$ of degree $(1,0)$ and $(0,1)$, respectively, and let $\ell_3$ be a~fibre of the~natural projection~\mbox{$E\to C$}. Then $\ell_1$, $\ell_2$, $\ell_3$ generate the~Mori cone $\overline{\mathrm{NE}}(X)$, and the~rays $\mathbb{R}_{\geqslant 0}[\ell_1]$ and $\mathbb{R}_{\geqslant 0}[\ell_2]$ give birational contractions $X\to U_1$ and $X\to U_2$, respectively.

Now, we let $L$ be a nef and big $\mathbb{Q}$-divisor on $X$. Then
$$
L\sim_{\mathbb{Q}} aH_1+bH_2+c(H_1+2H_2-E)
$$
for some non-negative rational numbers $a$, $b$, $c$. Then $L^3=3ab^2+12abc+6ac^2+3b^2c+9bc^2+5c^3$ and
$$
\frac{(-K_X)\cdot L^{2}}{L^3}=2\frac{3ab+4ac+b^2+6bc+5c^2}{3ab^2+12abc+6ac^2+3b^2c+9bc^2+5c^3}.
$$

\begin{lemma}
\label{lemma:3-21}
One has
$$
\beta_{L}(S)=-\frac{cf(a,b,c)}{(3ab^2+12abc+6ac^2+3b^2c+9bc^2+ 5c^3)^2},
$$
where
\begin{multline*}
\quad\quad \quad \quad \quad  f(a,b,c)=12a^2b^3+36a^2b^2c+36a^2bc^2+6ab^4+48ab^3c+\\
+114ab^2c^2+93abc^3+12ac^4+4b^3c^2+15b^2c^3+18bc^4+5c^5.\quad \quad \quad \quad
\end{multline*}
\end{lemma}

\begin{proof}
Take $u\in\mathbb{R}_{\geqslant 0}$. Then $L-uS$ is nef for $u\leqslant c$, and the divisor $L-uS$ is not pseudoeffective for $u>b+2c$. Moreover, if $u\in[c,b+2c]$, the positive part of the Zariski decomposition of the divisor $L-uS$ is $(a+c)H_1+(b+2c-u)H_2$, and its negative part is $(u-c)E$. Now, using \eqref{equation:beta}, we get the required formula for $\beta_L(S)$.
\end{proof}

Hence, if $L$ is ample, then $a$, $b$, $c$ are positive, so that $\beta_{L}(S)<0$ and $(X,L)$ is K-unstable.

\subsection{Family \textnumero 3.22}
\label{subsection:3-22}

Let $\mathrm{pr}_1\colon \mathbb{P}^1\times\mathbb{P}^2\to\mathbb{P}^1$ and $\mathrm{pr}_2\colon \mathbb{P}^1\times\mathbb{P}^2\to\mathbb{P}^2$ be the~projections to the~first and the~second factors, respectively, let $F_1$ be a~fibre of the~map $\mathrm{pr}_1$, let \mbox{$F_2=\mathrm{pr}_2^*(\mathcal{O}_{\mathbb{P}^2}(1))$},
let $C$ be a~conic in $F_1\simeq\mathbb{P}^2$, and let $\pi\colon X\to \mathbb{P}^1\times\mathbb{P}^2$ be the blow-up of the curve $C$.
Then $X$ is the unique smooth Fano threefold in Family \textnumero 3.22.  Let $E$ be the~exceptional surface of $\pi$, let $\widetilde{F}_1$ be the~strict transform on $X$ of the~surface $F_1$, let $H_1=\pi^*(F_1)$, let $H_2=\pi^*(F_2)$, let $S$ be~the~surface in $|2H_2|$ that contains $C$, and let $\widetilde{S}$ be  the~strict transform of this surface on $X$. Then all non-zero intersections of the divisors $H_1$, $H_2$, $E$ are $H_1\cdot H_2^2=1$, $H_2\cdot E^2=-2$ and $E^3=-4$, and there exists the~following commutative diagram:
$$
\xymatrix{
\mathbb{P}^1& &\mathbb{P}^1\times\mathbb{P}^2\ar@{->}[ll]_{\mathrm{pr}_1}\ar@{->}[rr]^{\mathrm{pr}_2} & &\mathbb{P}^2 \\
& & X\ar@{->}[drr]^{\phi}\ar@{->}[dll]_{\theta}\ar@{->}[u]_{\pi} & & \\
Y\ar@{->}[rr]_{\varphi}\ar@{->}[uu]^{\eta}& &\mathbb{P}(1,1,1,2) & &\mathbb{P}\big(\mathcal{O}_{\mathbb{P}^2}\oplus\mathcal{O}_{\mathbb{P}^2}(2)\big)\ar@{->}[ll]^{\vartheta}\ar@{->}[uu]_{\sigma}}
$$
where $\theta$ and $\phi$ are birational contractions of  $\widetilde{F}_1\cong\mathbb{P}^2$ and $\widetilde{S}\cong\mathbb{P}^1\times\mathbb{P}^1$ respectively, $\vartheta$ and $\varphi$ are birational contractions of $\phi({H}_1)$ and $\theta(\widetilde{S})$ respectively, $\sigma$ is a~$\mathbb{P}^1$-bundle, and $\eta$ is a~fibration into del Pezzo surfaces such that all its fibres except $\theta(\widetilde{S})$ are isomorphic to $\mathbb{P}^2$,
while $\theta(\widetilde{S})\cong\mathbb{P}(1,1,4)$. This commutative diagram shows that the~Mori cone $\overline{\mathrm{NE}}(X)$ is generated by the~extremal rays contracted by $\pi$, $\phi$ and $\theta$, and the nef cone of $X$ is generated by  $H_1$, $H_2$, $H_1+2H_2-E$.

Let $L$ be a nef and big $\mathbb{Q}$-divisor on $X$. Then
$$
L\sim_{\mathbb{Q}} aH_1+bH_2+c(H_1+2H_2-E)
$$
for some non-negative rational numbers $a$, $b$, $c$. Then $L^3=3ab^2+12abc+12ac^2+3b^2c+6bc^2+4c^3$ and
$$
\frac{(-K_X)\cdot L^{2}}{L^3}=2\frac{3ab+6ac+b^2+5bc+5c^2}{3ab^2+12abc+12ac^2+3b^2c+6bc^2+4c^3}.
$$
Note that $-K_X\sim 2H_1+3H_2-E$.

\begin{lemma}
\label{lemma:3-22}
One has
$$
\beta_{L}(\widetilde{F}_1)=-\frac{3cf(a,b,c)}{(3ab^2+12abc+12ac^2+3b^2c+6bc^2+ 4c^3)^2},
$$
where
\begin{multline*}
\quad \quad \quad \quad \quad f(a,b,c)=6a^2b^3+21a^2b^2c+24a^2bc^2+12a^2c^3+6ab^3c+\\
+16ab^2c^2+10abc^3+4ac^4+4b^3c^2+9b^2c^3+6bc^4+2c^5. \quad \quad \quad \quad  \quad
\end{multline*}
\end{lemma}

\begin{proof}
Take $u\in\mathbb{R}_{\geqslant 0}$. Then
$$
L-u\widetilde{F}_1\sim_{\mathbb{R}}\frac{2c-2a+b}{2}E+(a+c-u)\widetilde{F}_1+\frac{b+2c}{2}S,
$$
which implies that $L-u\widetilde{F}_1$ is pseudo-effective iff $u\leqslant a+c$. Moreover, if $u\in[0,c]$, the divisor $L-u\widetilde{F}_1$ is nef.
Furthermore, if $u\in[c,a+c]$, then the positive part of the Zariski decomposition of the divisor $L-u\widetilde{F}_1$ is $(a+c-u)H_1+(b+2c)H_2$, and its negative part is $(u-c)E$. Hence,
$$
\mathrm{vol}\big(L-u\widetilde{F}_1\big)=\left\{\aligned
&\big(L-u\widetilde{F}_1\big)^3\ \text{if $u\in[0,c]$},\\
&\big((a+c-u)H_1+(b+2c)H_2\big)^3\ \text{if $u\in[c,a+c]$},
\endaligned
\right.
$$
so
$$
\mathrm{vol}\big(L-u\widetilde{F}_1\big)=\left\{\aligned
&3ab^2+12abc+12ac^2+3b^2c-3b^2u+6bc^2-6bu^2+4c^3-4u^3\ \text{if $u\in[0,c]$},\\
&3(2c+b)^2(a+c-u)\ \text{if $u\in[c,a+c]$}.
\endaligned
\right.
$$
Now, using \eqref{equation:beta}, we obtain the required expression for $\beta_{L}(\widetilde{F}_1)$.
\end{proof}

Hence, if $L$ is ample, then $a,b,c$ are positive, so that $\beta_{L}(\widetilde{F}_1)<0$, and $(X,L)$ is K-unstable.

\subsection{Family \textnumero 3.23}
\label{subsection:3-23}

Let $\mathscr{C}$ be a~smooth conic in $\mathbb{P}^3$, let $P$ be a point in $\mathscr{C}$,
let $\phi\colon V_7\to\mathbb{P}^3$ be the~blow-up of the~point~$P$, let $C$ be the~proper transform on $V_7$ of the~conic $\mathscr{C}$,
and let $\pi\colon X\to V_7$ be the blow-up of the~curve $C$. Then $X$ is the unique smooth Fano threefold in Family \textnumero 3.23,
and we have the~following commutative diagram:
$$
\xymatrix{
\mathbb{P}^2& & \widehat{Q}\ar@{->}[ll]\ar@{->}[rr] & &Q \\
& & X\ar@{->}[drr]^{\varphi}\ar@{->}[dll]_{\pi}\ar@{->}[u]_{\psi} & & \\
V_7\ar@{->}[rr]_{\phi}\ar@{->}[uu]& &\mathbb{P}^3 & &\widetilde{\mathbb{P}}^3\ar@{->}[ll]^{\varpi}\ar@{->}[uu]}
$$
where $Q$ is a~smooth quadric threefold in $\mathbb{P}^4$,
the~morphism $\varpi$ is the~blow-up of the~conic~$\mathscr{C}$,
the~morphism~\mbox{$\widetilde{\mathbb{P}}^3\to Q$} is the~contraction of the~proper transform of the~plane in $\mathbb{P}^3$ that contains $\mathscr{C}$ to a~point,
$\varphi$ is a~blow-up of the~fibre of the~morphism $\varpi$ over the~point $P$,
the morphism $\widehat{Q}\to Q$ is the~blow-up of a~line on $Q$ that passes through the~latter point,
and $\widehat{Q}\to\mathbb{P}^2$ is a~$\mathbb{P}^1$-bundle.

Let $E_P$ the~$\varphi$-exceptional surface,
let $E_C$ be the~$\pi$-exceptional surface,
let $\Pi_C$ the~proper transform on $X$ of the~plane in $\mathbb{P}^3$ that contains $C$,
and let $H=(\phi\circ\pi)^*\big(\mathcal{O}_{\mathbb{P}^3}(1)\big)$.
Then $E_P$, $E_C$, $\Pi_C$ generate the cone of effective divisors on $X$ over $\mathbb{Q}$, $\Pi_C\sim H-E_P-E_C$ and
$$
-K_X\sim 4H-2E_P-E_C\sim 4\Pi_C+2E_P+3E_C.
$$
Moreover, the nef cone of $X$ is generated by $H$, $H-E_P$ and $2H-E_P-E_C$, because linear systems $|H|$, $|H-E_P|$ and $|2H-E_P-E_C|$ are base point free and give morphisms $X\to\mathbb{P}^3$, $X\to\mathbb{P}^2$ and $X\to Q$ in the diagram above, respectively.

Let $\mathbf{s}$ be~the~$(-1)$-curve in $\Pi_C\simeq\mathbb{F}_1$, let $\mathbf{f}$ be~a~fibre of the~natural projection~\mbox{$\Pi_C\to\mathbb{P}^1$}, and let $\ell$ be a~fibre of the~projection~$E_C\to C$ that is induced~by~$\pi$. Then the~curves $\mathbf{s}$, $\mathbf{f}$, $\ell$ generate the~Mori cone $\overline{\mathrm{NE}}(X)$, and the~contractions of the~corresponding extremal rays are $\varphi$, $\psi$, $\pi$, respectively.
Moreover, the non-zero intersections of the~curves $\mathbf{s}$, $\mathbf{f}$, $\ell$  with the~divisors $\Pi_C$, $E_P$, $E_C$, $H$
are
\[
E_C\cdot\mathbf{s}=-E_P\cdot\mathbf{s}=1\,\quad \ell\cdot\Pi_C=-\ell\cdot E_C=1,\quad \mathbf{f}\cdot E_P=\mathbf{f}\cdot E_C=\mathbf{f}\cdot H=-\mathbf{f}\cdot\Pi_C=1.
\]

Let $L$ be a nef and big $\mathbb{Q}$-divisor on $X$. Then
$$
L\sim_{\mathbb{Q}} aH+b(H-E_P)+c(2H-E_P-E_C),
$$
for some non-negative rational numbers $a$, $b$, $c$. All non-zero intersections of the divisors $H$, $E_P$, $E_C$ on $X$ are
$H^3=1$, $E_P^3=1$, $E_C^2\cdot H=-2$, $E_P\cdot E_C^2=-1$, $E_C^3=-4$. This gives
$$
L^3=a^3+3a^2b+6a^2c+3ab^2+12abc+6ac^2+3b^2c+6bc^2+2c^3
$$
and
$$
\frac{(-K_X)\cdot L^{2}}{L^3}=2\frac{2a^2+4ab+6ac+b^2+5bc+3c^2}{a^3+3a^2b+6a^2c+3ab^2+12abc+6ac^2+3b^2c+6bc^2+2c^3}
$$

\begin{lemma}
\label{lemma:3-23}
One has
$$
\beta_{L}(\Pi_C)=-\frac{f(a,b,c)}{g(a,b,c)},
$$
where
\begin{multline*}
f(a,b,c)=6a^5c+3a^4b^2+33a^4bc+27a^4c^2+4a^3b^3+72a^3b^2c+126a^3bc^2+48a^3c^3+54a^2b^3c+180a^2b^2c^2+\\
+180a^2bc^3+24a^2c^4+12ab^4c+78ab^3c^2+144ab^2c^3+84abc^4+8b^3c^3+18b^2c^4+12bc^5
\end{multline*}
and $g(a,b,c)=2(a^3+3a^2b+6a^2c+3ab^2+12abc+6ac^2+3b^2c+6bc^2+2c^3)^2$.
\end{lemma}

\begin{proof}
For real number $u\in\mathbb{R}_{\geqslant 0}$, we have
$$
L-u\Pi_C\sim_{\mathbb{R}}(a+b+2c-u)\Pi_C+(a+c)E_P+(c+a+b)E_C.
$$
This shows that $\tau(\Pi_C)=a+b+2c$. Moreover, for $u\in[0,a+b+2c]$, the Zariski decomposition of the divisor $L-u\Pi_C$ can be described as follows:
$$
P\big(L-u\Pi_C\big)=\left\{\aligned
&(a+b+2c-u)H-(b+c-u)E_P-(c-u)E_C\ \text{if $u\in[0,c]$},\\
&(a+b+2c-u)H-(b+c-u)E_P\ \text{if $u\in[c,b+c]$},\\
&(a+b+2c-u)H\ \text{if $u\in[b+c,a+b+2c]$},
\endaligned
\right.
$$
and
$$
N\big(L-u\Pi_C\big)=\left\{\aligned
&0\ \text{if $u\in[0,c]$},\\
&(u-c)E_C\ \text{if $u\in[c,b+c]$},\\
&(u-c)E_C+(u-b-c)E_P\ \text{if $u\in[b+c,a+b+2c]$}.\\
\endaligned
\right.
$$
Then $\mathrm{vol}\big(L-u\Pi_C\big)=\big(P(L-u\Pi_C)\big)^3$ for every $u\in[0,a+b+2c]$. Now, using \eqref{equation:beta} and integrating, we obtain the required expression for $\beta_L(\Pi_C)$.
\end{proof}

Hence, if $L$ is ample, then $a$, $b$ and $c$ are positive, so $\beta_L(\Pi_C)<0$ and $(X,L)$ is K-unstable.

\subsection{Family \textnumero 3.24}
\label{subsection:3-24}

Let $X$ be the unique smooth Fano threefold in Family \textnumero 3.24.
Then $X$ can be described as a blow-up $\eta\colon X\to\mathbb{P}^1\times\mathbb{P}^2$ of a smooth curve $C$ of degree~$(1,1)$ or as a blow-up $\pi\colon X\to W$ of the smooth divisor $W\subset \mathbb{P}^2\times\mathbb{P}^2$ along a curve $\ell$ of degree $(0,1)$. We have the following commutative diagram
$$
\xymatrix{
&&\mathbb{P}^1\times\mathbb{P}^2\ar@/_2pc/@{->}[ddll]_{\mathrm{pr}_1}\ar[rr]^{\mathrm{pr}_2}&&\mathbb{P}^{2}\\%
&&X\ar@{->}[u]_{\eta}\ar@{->}[dll]_{\zeta}\ar@{->}[d]^{\phi}\ar@{->}[r]^{\pi}&W\ar@{->}[dr]^{\omega_{1}}\ar@{->}[ur]_{\omega_{2}}&\\%
\mathbb{P}^{1}&&\mathbb{F}_{1}\ar@{->}^{\xi}[ll]\ar@{->}^{\gamma}[rr]&&\mathbb{P}^{2}}
$$
where $\omega_{1}$ and $\omega_2$ are natural $\mathbb{P}^{1}$-bundles,
$\gamma$ is the blow-up of the point $\omega_1(\ell)$, the morphisms $\phi$ and $\xi$ are $\mathbb{P}^{1}$-bundles,
$\zeta$ is a $\mathbb{F}_{1}$-bundle, $\mathrm{pr}_1$ and $\mathrm{pr}_2$ are projections to the first and the second factors, respectively.
This commutative diagram shows that the Mori cone $\overline{\mathrm{NE}}(X)$ is generated by the extremal rays
spanned by the curves contracted by $\pi$, $\eta$, $\phi$.

To describe the nef cone of the threefold $X$, we let $E$ be the $\pi$-exceptional surface, and set $H_1=(\omega_1\circ\pi)^*(\mathcal{O}_{\mathbb{P}^2}(1))$
and $H_2=(\omega_2\circ\pi)^*(\mathcal{O}_{\mathbb{P}^2}(1))$. Then the morphisms $\omega_1\circ\pi$, $\omega_2\circ\pi$ and $\zeta$ are given by the linear systems $|H_1|$, $|H_2|$ and $|H_1-E|$, respectively. This shows that the nef cone of $X$ is generated by the divisors $H_1$, $H_2$ and $H_1-E$. 

Let $L$ be a nef and big $\mathbb{Q}$-divisor on $X$. Then $L\sim_{\mathbb{Q}} aH_1+bH_2+c(H_1-E)$ for some $a,b,c\in\mathbb{Q}_{\geqslant 0}$, and 
$$
\frac{(-K_X)\cdot L^{2}}{L^3}=\frac{2(a^2+4ab+2ac+b^2+3bc)}{3b(a^2+ab+2ac+bc)}.
$$
Moreover, we compute
$$
\beta_{L}(E)=-2ac\frac{2a^2+3ab+b^2+bc}{3(a^2+ab+2ac+bc)^2}.
$$
Indeed, if $u\in\mathbb{R}_{\geqslant 0}$, then $L-uE$ is nef for $u\leqslant a$, and $L-uE$ is not pseudoeffective for $u>a+b$. Moreover, if $u\in[a,a+b]$, then the positive part of the Zariski decomposition of the divisor $L-uE$ is $(a+c)H_1+(a+b-u)H_2-(a+c)E$, and its negative part is $(u-a)S$ for the $\eta$-exceptional surface $S\sim H_2-E$. Now, using \eqref{equation:beta}, we get the required formula for $\beta_L(E)$.

Thus, if $L$ is ample, then $(X,L)$ is K-unstable.

\subsection{Family \textnumero 3.29}
\label{subsection:3-29}

Let $P$ be a point in $\mathbb{P}^3$, let $\phi\colon V_7\to\mathbb{P}^3$ be the blow-up of the point $P$, let $\overline{E}$ be the $\phi$-exceptional surface, let $\ell$ be a line in $\overline{E}\simeq\mathbb{P}^2$, and let $\pi\colon X\to V_7$ be the blow-up of the curve $\ell$ with exceptional surface $F$. Then $X$ is the unique smooth Fano threefold in Family \textnumero 3.29. Let $E$ be the strict transform on $X$ of the surface $\overline{E}$, let $S$ be the strict transform on $X$ of the plane in $\mathbb{P}^3$ that contains $P$ such that its strict transform on $V_7$  contains $\ell$. Then we have the following commutative diagram:
$$
\xymatrix@R=1em{
\mathbb{P}^2& &\mathbb{P}\big(\mathcal{O}_{\mathbb{P}^2}\oplus\mathcal{O}_{\mathbb{P}^2}(2)\big)\ar@{->}[ll]\ar@{->}[rr]^{\psi} & &\mathbb{P}(1,1,1,2) \\
& & & & \\
V_7\ar@{->}[drr]_{\phi}\ar@{->}[uu]& & X\ar@{->}[rr]^{\varphi}\ar@{->}[ll]_{\pi}\ar@{->}[uu]_{\theta} & &\widetilde{\mathbb{P}}^3\ar@{->}[uu]_{\vartheta}\ar@{->}[dll]^{\varpi}\ar@{->}[dll] \\
& &\mathbb{P}^3 & &}
$$
where $\theta$ is a contraction of the surface $S$ to a smooth curve, $\varphi$ is the contraction of the surface $E$ to a quotient singular point of type $\frac{1}{2}(1,1,1)$, $\vartheta$ is a contraction of the surface $\varphi(S)$ to a smooth curve, $\varpi$ is a weighted blow-up of the point $P$ with weights $(1,1,2)$, $\psi$ is the contraction of the surface $\theta(E)$ to the singular point of the threefold $\mathbb{P}(1,1,1,2)$, and both $V_7\to\mathbb{P}^2$ and $\mathbb{P}(\mathcal{O}_{\mathbb{P}^2}\oplus\mathcal{O}_{\mathbb{P}^2}(2))\to\mathbb{P}^2$ are $\mathbb{P}^1$-bundles.

To describe the Mori cone of $X$, let $C_1$ be the fibre of the natural projection $F\to\ell$, let $C_2$ be the strict transform on $X$ of a general line in $\mathbb{P}^3$ that is contained in $\phi\circ\pi(S)$ and passes through $P$, and let $C_3$ be a general line in $E\simeq\mathbb{P}^2$. Then the curves $C_1$, $C_2$, $C_3$ are contracted by $\pi$, $\theta$, $\varphi$, respectively. Thus, these curves generate the Mori cone of $X$.

Let $H$ be the strict transform on $X$ of a general plane in $\mathbb{P}^3$. Then in the diagram above, the morphisms $X\to\mathbb{P}^3$, $X\to\mathbb{P}^2$ and $X\to\mathbb{P}(1,1,1,2)$ are given by the linear systems $|H|$, $|H-E-F|$ and $|2H-E-2F|$, respectively. So, the nef cone of $X$ is generated by  $H$, $H-E-F$, $2H-E-2F$.

Let $L$ be a nef and big $\mathbb{Q}$-divisor on $X$. Then $L\sim_{\mathbb{Q}} aH+b(H-E-F)+c(2H-E-2F)$ for some $a,b,c\in\mathbb{Q}_{\geqslant 0}$. We have
$L^3=a^3+3a^2b+6a^2c+3ab^2+12abc+12ac^2+3b^2c+6bc^2+4c^3$ and
$$
\frac{(-K_X)\cdot L^{2}}{L^3}=2\frac{2a^2+4ab+8ac+b^2+5bc+5c^2}{a^3+3a^2b+6a^2c+3ab^2+12abc+12ac^2+3b^2c+6bc^2+4c^3}.
$$
To compute $\beta_{L}(E)$, take $u\in\mathbb{R}_{\geqslant 0}$ and consider the divisor $L-uE$. If $u\leqslant c$, then $L-uE$ is nef, so
$$
\mathrm{vol}(L-uE)=a^3+3a^2b+6a^2c+3ab^2+12abc+12ac^2+3b^2c-3b^2u+6bc^2-6bu^2+4c^3-4u^3.
$$
Similarly, if $u\in[c,a+c]$, then the negative part of the Zariski decomposition of the divisor $L-uE$ is $(u-c)F$, so its positive part is $(a+b+2c)H-(b+c+u)(E+F)$, which gives
$$
\mathrm{vol}(L-uE)=(a+c-u)(a^2+3ab+5ac+au+3b^2+9bc+3bu+7c^2+4cu+u^2).
$$
Finally, if $u>a+c$, then $L-uE$ is not pseudoeffective, so $\mathrm{vol}(L-uE)=0$. Using \eqref{equation:beta}, we get
$$
\beta_{L}(E)=-\frac{3f(a,b,c)}{2(a^3+3a^2b+6a^2c+3ab^2+12abc+12ac^2+3b^2c+6bc^2+4c^3)^2},
$$
where
\begin{multline*}
\quad \quad f(a,b,c)=3a^4b^2+9a^4bc+9a^4c^2+4a^3b^3+32a^3b^2c+54a^3bc^2+36a^3c^3+\\
+18a^2b^3c+72a^2b^2c^2+80a^2bc^3+40a^2c^4+18ab^3c^2+52ab^2c^3+
40abc^4+16ac^5+ 8b^3c^3+18b^2c^4+12bc^5+4c^6.
\end{multline*}
Hence, if $L$ is ample, then $a$, $b$ and $c$ are positive, so $\beta_L(E)<0$, and $(X,L)$ is K-unstable.

\subsection{Family \textnumero 4.8}
\label{subsection:4-8}
 This threefold can be described in two different geometric ways. First, as a blow-up of $\mathbb{P}^1\times\mathbb{P}^1\times\mathbb{P}^1$ along a smooth irreducible curve of degree $(0,1,1)$. Second, as a blow-up of the quadric cone $Q\subset\mathbb{P}^4$ at the disjoint union of its vertex and a smooth conic. To describe the geometry of $X$, let $\mathrm{pr}_i\colon \mathbb{P}^1\times\mathbb{P}^1\times\mathbb{P}^1\to\mathbb{P}^1$ be~the~projection to the~$i$-th factor, let $F_i$ be a~fibre of this~projection,
let $C$ be a~smooth curve of degree $(1,1)$ in $F_1\cong\mathbb{P}^1\times\mathbb{P}^1$, let $\pi\colon X\to \mathbb{P}^1\times\mathbb{P}^1\times\mathbb{P}^1$ be the blow-up of this curve, let $E$ be the $\pi$-exceptional surface, let $S$ be the strict transform on $X$ of the surface $F_1$, and let $R$ be the strict transform on $X$ of the surface in $|F_2+F_3|$ that contains $C$. Then $S\simeq R\simeq \mathbb{P}^1\times\mathbb{P}^1$, the surfaces $S$ and $R$ are disjoint, $S\vert_{S}$ is a divisor of degree $(-1,-1)$, and $R\vert_{R}$ is a divisor of degree $(-1,2)$. Moreover, we have the following commutative diagram:
$$
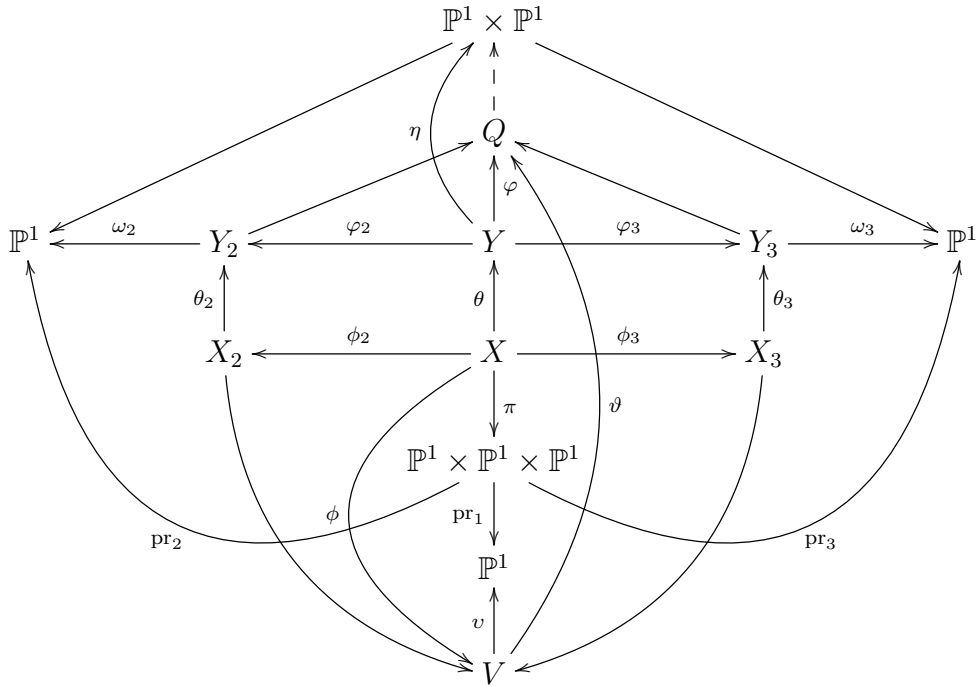
\xymatrix{
&&&&\mathbb{P}^1\times\mathbb{P}^1\ar[rrrrdd]\ar[lllldd]&&&&\\
&&&&Q\ar@{-->}[u]&&&&\\
\mathbb{P}^1&&Y_{2}\ar[ll]_{\omega_{2}}\ar[urr]&&Y\ar@/^2pc/[uu]^{\eta}\ar[u]_{\varphi}\ar[ll]_{\varphi_{2}}\ar[rr]^{\varphi_{3}}&&Y_{3}\ar[ull]\ar[rr]^{\omega_{3}}&&\mathbb{P}^1\\
&&X_{2}\ar@/_30pt/[dddrr]\ar[u]^{\theta_2}&&X\ar@/_55pt/[ddd]_{\phi}\ar[ll]_{\phi_{2}}\ar[rr]^{\phi_{3}}\ar[d]^{\pi}\ar[u]^{\theta}&&X_{3}\ar@/^30pt/[dddll]\ar[u]_{\theta_3}&&\\
&&&&\mathbb{P}^1\times\mathbb{P}^1\times\mathbb{P}^1\ar[d]_{\mathrm{pr}_1}\ar@/^6pc/[uullll]^{\mathrm{pr}_2}\ar@/_6pc/[uurrrr]_{\mathrm{pr}_3}&& &&\\
&&&&\mathbb{P}^1&&&&\\
&&&&V\ar[u]^{\upsilon}\ar@/_40pt/[uuuuu]_{\vartheta}&&&&}
$$
where $\phi$ is a contraction of the surface $S$ to an ordinary double point of the threefold $V$, $\phi_2$ and $\phi_3$ are contractions of the surface $S$ to smooth curves such that $\phi_3\circ\phi_2^{-1}$ is an Atiyah flop, both morphisms $X_2\to V$ and $X_3\to V$ are small resolutions of the threefold $V$, $\theta$, $\theta_2$, $\theta_3$ are contractions to smooth curves of the surfaces $R$, $\phi_2(R)$, $\phi_3(R)$, respectively, $\varphi$ is a contraction of the surface $\theta(S)$ to the vertex of the quadric cone $Q$, $\varphi_2$ and $\varphi_3$ are contractions of the surface $\theta(S)$ to smooth curves such that $\varphi_3\circ\varphi_2^{-1}$ is an Atiyah flop, both morphisms $Y_2\to Q$ and $Y_3\to Q$ are small resolutions of the quadric cone $Q$, $\eta$ is a $\mathbb{P}^1$-bundle, $\omega_2$ and $\omega_3$ are $\mathbb{P}^2$-bundles, $\vartheta$ is the contraction of the surface $\phi(R)$ to a smooth conic in $Q$, $\upsilon$ is a fibration into quadric surfaces, and both $\mathbb{P}^1\times\mathbb{P}^1\to\mathbb{P}^1$ are natural projections.
Using this commutative diagram, we see that the Mori cone of the threefold $X$ is simplicial \cite{Matsuki}, and contractions of its extremal rays are given by birational morphisms $\pi$, $\phi_2$, $\phi_3$ and $\theta$. Namely, let $\ell_1$ be an irreducible curve in $E$ contracted by $\pi$, let $\ell_2$ and $\ell_3$ be irreducible curves in $S$ that are contracted by $\phi_2$ and $\phi_3$, respectively, and let $\ell_4$ be an irreducible curve contracted by $\theta$. Then $\ell_1$, $\ell_2$, $\ell_3$, $\ell_4$ generate the Mori cone of $X$.

To describe the nef cone of $X$, we let $H_i=\pi^*(F_i)$. Then the nef cone of $X$ is generated by $H_1$, $H_2$, $H_3$ and $H_1+H_2+H_3-E$. Note that $|H_1+H_2+H_3-E|$ is base point free and gives the morphism $\varphi\circ\theta\colon X\to Q$ described above. Note also that $S\sim H_1-E$ and $R\sim H_2+H_3-E$.

Let $L$ be a nef and big $\mathbb{Q}$-divisor on $X$. Then there are $a_1,a_2,a_3,a_4\in\mathbb{Q}_{\geqslant 0}$ such that
$$
L\sim_{\mathbb{Q}} a_1H_1+a_2H_2+a_3H_3+a_4(H_1+H_2+H_3-E).
$$
We have
$L^3=6a_1a_2a_3+6a_1a_2a_4+6a_1a_3a_4+6a_1a_4^2+6a_2a_3a_4+3a_2a_4^2+3a_3a_4^2+2a_4^3$,
and
$$
\frac{(-K_X)\cdot L^{2}}{L^3}=2\frac{2a_1a_2+2a_1a_3+4a_1a_4+2a_2a_3+3a_2a_4+3a_3a_4+3a_4^2}{6a_1a_2a_3+6a_1a_2a_4+6a_1a_3a_4+6a_1a_4^2+6a_2a_3a_4+3a_2a_4^2+3a_3a_4^2+2a_4^3}.
$$
Set $g(a_1,a_2,a_3,a_4)=(6a_1a_2a_3+6a_1a_2a_4+6a_1a_3a_4+6a_1a_4^2+6a_2a_3a_4+3a_2a_4^2+3a_3a_4^2+2a_4^3)^2$.

\begin{lemma}
\label{lemma:4-8}
One has
$$
\beta_{L}(S)=-a_4\frac{f_1(a_1,a_2,a_3,a_4)}{g(a_1,a_2,a_3,a_4)},
$$
where
\begin{multline*}
f_1(a_1,a_2,a_3,a_4)=18a_1^2a_2^2a_3+18a_1^2a_2a_3^2+54a_1^2a_2a_3a_4+12a_1^2a_2a_4^2+12a_1^2a_3a_4^2+\\
+6a_1^2a_4^3+18a_1a_2^2a_3a_4-6a_1a_2^2a_4^2+18a_1a_2a_3^2a_4+48a_1a_2a_3a_4^2-6a_1a_3^2a_4^2+\\
+12a_2^2a_3a_4^2+12a_2a_3^2a_4^2+24a_2a_3a_4^3+3a_2a_4^4+3a_3a_4^4+a_4^5.
\end{multline*}
Moreover, if $a_2\leqslant a_3$, then
$$
\beta_{L}(E)=-\frac{f_2(a_1,a_2,a_3,a_4)}{g(a_1,a_2,a_3,a_4)},
$$
where
\begin{multline*}
f_2(a_1,a_2,a_3,a_4)=12a_1^2a_2^4-24a_1^2a_2^3a_3+24a_1^2a_2^3a_4-54a_1^2a_2^2a_3a_4+18a_1^2a_2a_3^2a_4-\\
-18a_1^2a_2a_3a_4^2-12a_1^2a_2a_4^3-12a_1^2a_3a_4^3-6a_1^2a_4^4+18a_1a_2^4a_4-36a_1a_2^3a_3a_4+30a_1a_2^3a_4^2+\\
+18a_1a_2^2a_3^2a_4-54a_1a_2^2a_3a_4^2+6a_1a_2^2a_4^3+36a_1a_2a_3^2a_4^2+6a_1a_3^2a_4^3+12a_2^4a_4^2-\\
-24a_2^3a_3a_4^2+14a_2^3a_4^3-36a_2^2a_3a_4^3+6a_2a_3^2a_4^3-12a_2a_3a_4^4-3a_2a_4^5-3a_3a_4^5-a_4^6.
\end{multline*}
\end{lemma}

\begin{proof}
For $u\in[0,a_1+a_4]$, the positive part of the Zariski decomposition of the divisor $L-uS$ is
$$
P\big(L-uS\big)=\left\{\aligned
&(a_1+a_4-u)H_1+(a_2+a_4)H_2+(a_3+a_4)H_3-(a_4-u)E\ \text{if $u\in[0,a_4]$},\\
&(a_1+a_4-u)H_1+(a_2+a_4)H_2+(a_3+a_4)H_3\ \text{if $u\in[a_4,a_1+a_4]$},
\endaligned
\right.
$$
and its negative part is
$$
N\big(L-uS\big)=\left\{\aligned
&0\ \text{if $u\in[0,a_4]$},\\
&(u-a_4)E\ \text{if $u\in[a_4,a_1+a_4]$}.
\endaligned
\right.
$$
Then $\mathrm{vol}(L-uS)=(P(L-uS))^3$ for every $u\in[0,a_1+a_4]$.
For $u>a_1+a_4$, the divisor $L-uS$ is not pseudoeffective, so $\mathrm{vol}(L-uS)=0$.
Now, using \eqref{equation:beta} we obtain the expression for $\beta_L(S)$.

Similarly, if $u\geqslant a_1+a_2+a_4$, the divisor $L-uE$ is not big, so $\mathrm{vol}(L-uE)=0$.
If $a_1\leqslant a_2\leqslant a_3$ and $u\in[0,a_1+a_2+a_4]$, the positive part of the Zariski decomposition of the divisor $L-uE$ is
$$
P\big(L-uE\big)=\left\{\aligned
&(a_1+a_4)H_1+(a_2+a_4)H_2+(a_3+a_4)H_3-(a_4+u)E\ \text{if $u\in[0,a_1]$},\\
&(a_1+a_4)(H_1-E)+(a_1+a_2+a_4-u)H_2+(a_1+a_3+a_4-u)H_3\ \text{if $u\in[a_1,a_2]$},\\
&(a_1+a_2+a_4-u)(H_1+H_2-E)+(a_1+a_3+a_4-u)H_3\ \text{if $u\in[a_2,a_1+a_2+a_4]$},
\endaligned
\right.
$$
and its negative part is
$$
N\big(L-uE\big)=\left\{\aligned
&0\ \text{if $u\in[0,a_1]$},\\
&(u-a_1)R\ \text{if $u\in[a_1,a_2]$},\\
&(u-a_1)R+(u-a_2)S\ \text{if $u\in[a_2,a_1+a_2+a_4]$}.
\endaligned
\right.
$$
If $a_1\geqslant a_2$, $a_2\leqslant a_3$ and $u\in[0,a_1+a_2+a_4]$, the positive part of the Zariski decomposition of the divisor $L-uE$ is
$$
P\big(L-uE\big)=\left\{\aligned
&(a_1+a_4)H_1+(a_2+a_4)H_2+(a_3+a_4)H_3-(a_4+u)E\ \text{if $u\in[0,a_2]$},\\
&(a_1+a_2+a_4-u)H_1+(a_2+a_4)H_2+(a_3+a_4)H_3-(a_2+a_4)E\ \text{if $u\in[a_2,a_1]$},\\
&(a_1+a_2+a_4-u)(H_1+H_2-E)+(a_1+a_3+a_4-u)H_3\ \text{if $u\in[a_1,a_1+a_2+a_4]$},
\endaligned
\right.
$$
and its negative part is
$$
N\big(L-uE\big)=\left\{\aligned
&0\ \text{if $u\in[0,a_2]$},\\
&(u-a_2)S\ \text{if $u\in[a_2,a_1]$},\\
&(u-a_1)R+(u-a_2)S\ \text{if $u\in[a_1,a_1+a_2+a_4]$}.
\endaligned
\right.
$$
Now, using \eqref{equation:beta} we obtain the required expression for $\beta_L(E)$.
\end{proof}

\begin{lemma}
\label{lemma:4-8-unstable}
Suppose that $L$ is ample. Then $\beta_L(S)<0$ or $\beta_L(E)<0$.
\end{lemma}

\begin{proof}
Without loss of generality, we may assume that $a_2\leqslant a_3$. Since $L$ is ample, all $a_1$, $a_2$, $a_3$, $a_4$ are positive. Suppose that $\beta_L(S)\geqslant 0$ and $\beta_L(E)\geqslant 0$. Then, in the notations of Lemma~\ref{lemma:4-8}, we have $f_1(a_1,a_2,a_3,a_4)\leqslant 0$ and $f_2(a_1,a_2,a_3,a_4)\leqslant 0$. However, if $a_1\leqslant a_2$, then
\begin{multline*}
0\geqslant f_1(a_1,a_2,a_3,a_4)=36a_1^5+90a_1^4a_4+54a_1^4\delta_2+108a_1^4\delta_1+84a_1^3a_4^2+108a_1^3a_4\delta_2+216a_1^3a_4\delta_1+\\
+18a_1^3\delta_2^2+108a_1^3\delta_2\delta_1+108a_1^3\delta_1^2+30a_1^2a_4^3+84a_1^2a_4^2\delta_2+168a_1^2a_4^2\delta_1+\\
+18a_1^2a_4\delta_2^2+162a_1^2a_4\delta_2\delta_1+162a_1^2a_4\delta_1^2+18a_1^2\delta_2^2\delta_1+54a_1^2\delta_2\delta_1^2+\\
+36a_1^2\delta_1^3+6a_1a_4^4+24a_1a_4^3\delta_2+48a_1a_4^3\delta_1+6a_1a_4^2\delta_2^2+108a_1a_4^2\delta_2\delta_1+108a_1a_4^2\delta_1^2+\\
+18a_1a_4\delta_2^2\delta_1+54a_1a_4\delta_2\delta_1^2+36a_1a_4\delta_1^3+a_4^5+3a_4^4\delta_2+6a_4^4\delta_1+\\
+24a_4^3\delta_2\delta_1+24a_4^3\delta_1^2+ 12a_4^2\delta_2^2\delta_1+36a_4^2\delta_2\delta_1^2+24a_4^2\delta_1^3>0,
\end{multline*}
where $\delta_1=a_2-a_1$ and $\delta_2=a_3-a_2$. Thus, we see that $a_1>a_2$. If $a_1\leqslant a_3$, then
\begin{multline*}
0\geqslant (a_2+a_4)f_1(a_1,a_2,a_3,a_4)+f_2(a_1,a_2,a_3,a_4)=24a_2^6+114a_2^5a_4+30a_2^5\epsilon_2+78a_2^5\epsilon_1+\\
+156a_2^4a_4^2+144a_2^4a_4\epsilon_2+336a_2^4a_4\epsilon_1+18a_2^4\epsilon_2^2+96a_2^4\epsilon_2\epsilon_1+102a_2^4\epsilon_1^2+86a_2^3a_4^3+\\
+168a_2^3a_4^2\epsilon_2+372a_2^3a_4^2\epsilon_1+72a_2^3a_4\epsilon_2^2+378a_2^3a_4\epsilon_2\epsilon_1+384a_2^3a_4\epsilon_1^2+36a_2^3\epsilon_2^2\epsilon_1+\\
+102a_2^3\epsilon_2\epsilon_1^2+66a_2^3\epsilon_1^3+18a_2^2a_4^4+84a_2^2a_4^3\epsilon_2+144a_2^2a_4^3\epsilon_1+60a_2^2a_4^2\epsilon_2^2+324a_2^2a_4^2\epsilon_2\epsilon_1+\\
+324a_2^2a_4^2\epsilon_1^2+108a_2^2a_4\epsilon_2^2\epsilon_1+306a_2^2a_4\epsilon_2\epsilon_1^2+198a_2^2a_4\epsilon_1^3+18a_2^2\epsilon_2^2\epsilon_1^2+36a_2^2\epsilon_2\epsilon_1^3+\\
+18a_2^2\epsilon_1^4+a_2a_4^5+15a_2a_4^4\epsilon_2+15a_2a_4^4\epsilon_1+18a_2a_4^3\epsilon_2^2+84a_2a_4^3\epsilon_2\epsilon_1+72a_2a_4^3\epsilon_1^2+48a_2a_4^2\epsilon_2^2\epsilon_1+\\
+144a_2a_4^2\epsilon_2\epsilon_1^2+96a_2a_4^2\epsilon_1^3+36a_2a_4\epsilon_2^2\epsilon_1^2+72a_2a_4\epsilon_2\epsilon_1^3+36a_2a_4\epsilon_1^4>0,
\end{multline*}
where $\epsilon_1=a_1-a_2$ and $\epsilon_2=a_3-a_1$. Thus, we see that $a_1>a_3$. Then
\begin{multline*}
0\geqslant (a_2+a_4)f_1(a_1,a_2,a_3,a_4)+f_2(a_1,a_2,a_3,a_4)=24a_2^6+114a_2^5a_4+48a_2^5\varepsilon_1+78a_2^5\varepsilon_2+\\
+156a_2^4a_4^2+192a_2^4a_4\varepsilon_1+336a_2^4a_4\varepsilon_2+24a_2^4\varepsilon_1^2+108a_2^4\varepsilon_1\varepsilon_2+102a_2^4\varepsilon_2^2+86a_2^3a_4^3+204a_2^3a_4^2\varepsilon_1+\\
+372a_2^3a_4^2\varepsilon_2+78a_2^3a_4\varepsilon_1^2+390a_2^3a_4\varepsilon_1\varepsilon_2+384a_2^3a_4\varepsilon_2^2+30a_2^3\varepsilon_1^2\varepsilon_2+96a_2^3\varepsilon_1\varepsilon_2^2+66a_2^3\varepsilon_2^3+\\
+18a_2^2a_4^4+60a_2^2a_4^3\varepsilon_1+144a_2^2a_4^3\varepsilon_2+60a_2^2a_4^2\varepsilon_1^2+324a_2^2a_4^2\varepsilon_1\varepsilon_2+324a_2^2a_4^2\varepsilon_2^2+90a_2^2a_4\varepsilon_1^2\varepsilon_2+288a_2^2a_4\varepsilon_1\varepsilon_2^2+\\
+198a_2^2a_4\varepsilon_2^3+18a_2^2\varepsilon_1^2\varepsilon_2^2+36a_2^2\varepsilon_1\varepsilon_2^3+18a_2^2\varepsilon_2^4+a_2a_4^5+15a_2a_4^4\varepsilon_2+6a_2a_4^3\varepsilon_1^2+60a_2a_4^3\varepsilon_1\varepsilon_2+\\
+72a_2a_4^3\varepsilon_2^2+48a_2a_4^2\varepsilon_1^2\varepsilon_2+144a_2a_4^2\varepsilon_1\varepsilon_2^2+96a_2a_4^2\varepsilon_2^3+36a_2a_4\varepsilon_1^2\varepsilon_2^2+ 72a_2a_4\varepsilon_1\varepsilon_2^3+36a_2a_4\varepsilon_2^4>0,
\end{multline*}
where $\varepsilon_1=a_1-a_3$ and $\varepsilon_2=a_3-a_2$. This is a contradiction.
\end{proof}

Thus, if $L$ is ample, then $(X,L)$ is K-unstable.

\subsection{Family \textnumero 4.9}
\label{subsection:4-9}

Let $\ell_1$ and $\ell_2$ be two skew lines in $\mathbb{P}^3$, let $P$ be a point in $\ell_1$, let $h\colon V_7\to\mathbb{P}^3$ be the blow-up of the point $P$, and let $f\colon X\to V_7$ be the blow-up of the strict transforms of the lines $\ell_1$ and $\ell_2$. Then $X$ is the unique smooth Fano threefold in Family \textnumero 4.9. To describe the cone of nef divisors of the threefold $X$, let $F_1$ and $F_2$ be $f$-exceptional surfaces such that $h\circ f(F_1)=\ell_1$ and $h\circ f(F_2)=\ell_2$, let $E$ be the strict transform on $X$ of the $h$-exceptional surface, and let $S$ be the strict transform on $X$ of the plane in $\mathbb{P}^3$ that contains $P$ and $\ell_2$, and let $H$ be the strict transform on $X$ of a general plane in $\mathbb{P}^3$. Then the nef cone of $X$ is generated by the divisors $H$, $H-E$, $H-E-F_1$ and $H-F_2$. Note that $-K_X\sim 4H-2E-F_1-F_2$.

To describe the Mori cone of the threefold $X$, let $C_1$ be a smooth irreducible curve in $E\simeq\mathbb{F}_1$ such that $C_1^2=0$, let $C_2$ be a general fibre of the natural projection $F_1\to\ell_1$, let $C_3$ be a general fibre of the natural projection $F_2\to\ell_2$, and let $C_4$ be the strict transform on $X$ of the general line in $\mathbb{P}^3$ that contains $P$ and intersects $\ell_2$. Then the Mori cone of the threefold $X$ is generated by the curves $C_1$, $C_2$, $C_3$, $C_4$. Note that $C_4\subset S$.

Now, let $L$ be a nef and big $\mathbb{Q}$-divisor on $X$. Then
$$
L\sim_{\mathbb{Q}} aH+b(H-E)+c(H-F_1-E)+d(H-F_2)
$$
for some non-negative rational numbers $a$, $b$, $c$, $d$. We have
$$
L^3=a^3+3a^2b+3a^2c+3a^2d+3ab^2+6abc+6abd+6acd+3b^2d+6bcd
$$
and
$$
\frac{(-K_X)\cdot L^{2}}{L^3}=2\frac{2a^2+4ab+3ac+3ad+b^2+2bc+3bd+2cd}{a^3+3a^2b+3a^2c+3a^2d+3ab^2+6abc+6abd+6acd+3b^2d+6bcd}.
$$
Now, using \eqref{equation:beta} we compute $\beta_L(E)$, $\beta_L(S)$, $\beta_L(F_1)$, $\beta_L(F_2)$ as follows. Set
$$
g(a,b,c,d)=2(a^3+3a^2b+3a^2c+3a^2d+3ab^2+6abc+6abd+6acd+3b^2d+6bcd)^2.
$$
Then
$$
\beta_{L}(E)=-3a\frac{f_1(a,b,c,d)}{g(a,b,c,d)}
$$
\begin{multline*}
\text{with }f_1(a,b,c,d)=a^4c-a^4d+3a^3b^2+8a^3bc-5a^3bd+4a^2b^3+12a^2b^2c
+2a^2b^2d+8a^2bc^2+\\12a^2bcd-4a^2c^2d+4a^2cd^2+6ab^3d+18ab^2cd+6ab^2d^2
+12abc^2d+18abcd^2+6b^3d^2+18b^2cd^2+12bc^2d^2.
\end{multline*}
Similarly, we have
$$
\beta_{L}(S)=-\frac{f_{2}(a,b,c,d)}{g(a,b,c,d)}
$$
\begin{multline*}
\text{where }f_2(a,b,c,d)=-3a^5c+3a^5d+3a^4b^2-12a^4bc+15a^4bd+4a^3b^3-24a^3b^2c+\\
+42a^3b^2d+12a^3bc^2+12a^3bcd+12a^3c^2d-12a^3cd^2-24a^2b^3c+42a^2b^3d+18a^2b^2cd+\\
+18a^2b^2d^2+72a^2bc^2d-18a^2bcd^2+12ab^4d+18ab^3d^2+36ab^2c^2d-18ab^2cd^2+36abc^2d^2-24b^3cd^2.
\end{multline*}
Finally, we have
$$
\beta_{L}(F_1)=-\beta_{L}(F_2)=\frac{f_{3}(a,b,c,d)}{g(a,b,c,d)}
$$
\begin{multline*}
\text{where }f_3(a,b,c,d)=6a^5c-6a^5d+6a^4b^2+36a^4bc-30a^4bd+8a^3b^3+60a^3b^2c-36a^3b^2d+\\
+12a^3bc^2+24a^3bcd-24a^3c^2d+24a^3cd^2+24a^2b^3c-24a^2b^3d+36a^2b^2cd-36a^2bc^2d+\\
+72a^2bcd^2-12ab^4d-36ab^2c^2d+ 72ab^2cd^2+24b^3cd^2.
\end{multline*}
Observe that $f_1=f_2+f_3$.

\begin{lemma}
\label{lemma:4-9}
Suppose that $L$ is ample. Then $\beta_L(E)<0$ or $\beta_L(S)<0$.
\end{lemma}

\begin{proof}
Since $L$ is ample, all $a$, $b$, $c$, $d$ are positive. Suppose that $\beta_L(E)\geqslant 0$ and $\beta_L(S)\geqslant 0$. Then $f_1(a,b,c,d)\leqslant 0$ and $f_2(a,b,c,d)\leqslant 0$. However, if $b\geqslant a$, then
\begin{multline*}
0\geqslant f_1(a,b,c,d)=21a^6+63a^5c+6a^5d+54a^5\delta+24a^4c^2+90a^4cd+96a^4c\delta+36a^4d^2+51a^4d\delta
+45a^4\delta^2+\\24a^3c^2d+24a^3c^2\delta+120a^3cd^2+144a^3cd\delta+36a^3c\delta^2+90a^3d^2\delta+60a^3d\delta^2+12a^3\delta^3+36a^2c^2d^2+36a^2c^2d\delta\\+162a^2cd^2\delta+54a^2cd\delta^2
+72a^2d^2\delta^2+18a^2d\delta^3+36ac^2d^2\delta+54acd^2\delta^2+18ad^2\delta^3>0,
\end{multline*}
where $\delta=b-a$. Similarly, if $b\leqslant a$, then
\begin{multline*}
0\geqslant f_1(a,b,c,d)+f_2(a,b,c,d)=28b^6+120b^5d+96b^5\epsilon+36b^4c^2+120b^4cd+36b^4c\epsilon
+72b^4d^2\\+276b^4d\epsilon+120b^4\epsilon^2+144b^3c^2d+108b^3c^2\epsilon+48b^3cd^2+288b^3cd\epsilon+84b^3c\epsilon^2
+108b^3d^2\epsilon+204b^3d\epsilon^2+64b^3\epsilon^3\\+72b^2c^2d^2+252b^2c^2d\epsilon+108b^2c^2\epsilon^2+108b^2cd^2\epsilon
+216b^2cd\epsilon^2+60b^2c\epsilon^3+36b^2d^2\epsilon^2+48b^2d\epsilon^3+12b^2\epsilon^4+\\72bc^2d^2\epsilon
+108bc^2d\epsilon^2+ 36bc^2\epsilon^3+36bcd^2\epsilon^2+48bcd\epsilon^3+12bc\epsilon^4>0,
\end{multline*}
where $\epsilon=a-b$. The contradiction obtained completes the proof of the lemma.
\end{proof}

Thus, if $L$ is ample, then $(X,L)$ is K-unstable.

\subsection{Family \textnumero 4.10}
\label{subsection:4-10}

Let $X=\mathbb{P}^1\times S_7$, where $S_7$ is the smooth del Pezzo surface of degree~$7$.
Then $X$ is the unique smooth Fano threefold in Family \textnumero 4.10.
Let $\pi_1\colon X\to\mathbb{P}^1$ and $\pi_2\colon X\to S_7$ be the projections to the first and the second factors, respectively.
Let $F_1$ and $F_2$ be irreducible surfaces in $X$ such that $\pi_2(F_1)$ and $\pi_2(F_2)$ are disjoint $(-1)$-curves,
let $E$ be the surface in $X$ such that $\pi_2(E)$ is the $(-1)$-curve that intersects both $\pi_2(F_1)$ and $\pi_2(F_2)$,
and let $G$ be a fibre of the projection $\pi_1$. Then $-K_X\sim 3E+2(F_1+F_2)+2G$,
and all non-zero intersections of the divisors $F_1$, $F_2$, $E$ and $G$ are
$G\cdot F_1^2=-1$, $G\cdot F_2^2=-1$, $G\cdot E^2=-1$, $E\cdot F_1\cdot G=1$, $E\cdot F_2\cdot G=1$.

Let $L$ be a nef and big $\mathbb{Q}$-divisor on $X$.  Then
$$
L\sim_{\mathbb{Q}}a_1(E+F_1)+a_2(E+F_2)+b(E+F_1+F_2)+cG
$$
for some non-negative rational numbers $a_1$, $a_2$, $b$, $c$. We have
$$
L^3=3c(2a_1a_2+2a_1b+2a_2b+b^2).
$$
and
$$
\frac{(-K_X)\cdot L^2}{L^3}=\frac{2(2a_1a_2+2a_1b+2a_1c+2a_2b+2a_2c+b^2+3bc)}{3c(2a_1a_2+2a_1b+2a_2b+b^2)}.
$$

Let us compute $\beta_L(E)$. Without loss of generality, we may assume~$a_1\leqslant a_2$.
Fix $u\in\mathbb{R}_{\geqslant 0}$. Then
$$
L-uE\sim_{\mathbb{R}}(a_1+a_2+b-u)E+(a_1+b)F_1+(a_2+b)F_2+cG,
$$
so that $L-uE$ is not pseudo-effective for $u>a_1+a_2+b$.
Moreover, if $u\leqslant a_1$, then $L-uE$ is~nef.
Furthermore, if $a_1<u\leqslant a_2$, then the~Zariski decomposition of the divisor $L-uE$ is
$$
L-uE\sim_{\mathbb{R}}\underbrace{(a_1+a_2+b-u)(E+F_2)+(a_1+b)F_1+cG}_{\text{positive part}}+\underbrace{(u-a_1)F_2}_{\text{negative part}}.
$$
Finally, if $a_2<u\leqslant a_1+a_2+b$, then the~Zariski decomposition of the divisor $L-uE$ is
$$
L-uE\sim_{\mathbb{R}}\underbrace{(a_1+a_2+b-u)(E+F_1+F_2)+cG}_{\text{positive part}}+\underbrace{(u-a_2)F_1+(u-a_1)F_2}_{\text{negative part}}.
$$
Hence, if $0\leqslant u\leqslant a_1+a_2+b$, then
$$
\mathrm{vol}\big(L-uE\big)=
\left\{\aligned
&\big((a_1+a_2+b-u)E+(a_1+b)F_1+(a_2+b)F_2+cG\big)^3\ \text{if $0\leqslant u\leqslant a_1$}, \\
&\big((a_1+a_2+b-u)(E+F_2)+(a_1+b)F_1+cG\big)^3\ \text{if $a_1<u\leqslant a_2$},\\
&\big((a_1+a_2+b-u)(E+F_1+F_2)+cG\big)^3\ \text{if $a_2<u\leqslant a_1+a_2+b$}.
\endaligned
\right.
$$
We compute
$$
\mathrm{vol}\big(L-uE\big)=
\left\{\aligned
&3c(2a_1a_2+2a_1b+2a_2b+b^2-2bu-u^2)\ \text{if $0\leqslant u\leqslant a_1$}, \\
&3c(a_1+b)(a_1+2a_2+b-2u)\ \text{if $a_1<u\leqslant a_2$},\\
&3c(a_1+a_2+b-u)^2\ \text{if $a_2<u\leqslant a_1+a_2+b$}.
\endaligned
\right.
$$
Therefore, using \eqref{equation:beta} and integrating, we get
$$
\beta_L(E)=-\frac{2b(3a_1^2a_2+3a_1a_2^2+6a_1a_2b+a_1b^2+a_2b^2)}{3(2a_1a_2+2a_1b+2a_2b+b^2)^2}.
$$
Thus, if $L$ is ample, then $\beta_{L}(E)<0$, so $(X,L)$ is K-unstable.

\subsection{Family \textnumero 4.11}
\label{subsection:4-11}

Let $\mathbf{s}$ be the $(-1)$-curve in the surface $\mathbb{F}_1$, let $\mathbf{f}$ be a fibre of the natural projection $\mathbb{F}_1\to\mathbb{P}^1$, let $\mathrm{pr}_1\colon\mathbb{P}^1\times\mathbb{F}_1\to \mathbb{P}^1$ be the projection to the first factor, let $\mathrm{pr}_2\colon\mathbb{P}^1\times\mathbb{F}_1\to \mathbb{F}_1$ be the projection to the second factor, let $\overline{A}$ be a fibre of $\mathrm{pr}_1$, let $\overline{S}=\mathrm{pr}_2^*(\mathbf{s})$, let $\overline{F}=\mathrm{pr}_2^*(\mathbf{f})$, let $Z=\overline{A}\cap\overline{S}$, let $\pi\colon X\to\mathbb{P}^1\times\mathbb{F}_1$ be the blow-up of the curve $Z$, let $E$ be the $\pi$-exceptional surface, let $A=\pi^*(\overline{A})$, let $S=\pi^*(\overline{S})$, let $F=\pi^*(\overline{F})$, let $\widetilde{A}$, $\widetilde{S}$, $\widetilde{F}$ be the strict transforms on $X$ of the surfaces $\overline{A}$, $\overline{S}$, $\overline{F}$, respectively. Then $X$ is the unique smooth Fano threefold in Family \textnumero 4.11, $\widetilde{A}\sim A-E$, $\widetilde{S}\sim S-E$, $\widetilde{F}=F$, $\widetilde{A}\simeq\overline{A}\simeq\mathbb{F}_1$, $\widetilde{S}\simeq\overline{S}\simeq\mathbb{P}^1\times\mathbb{P}^1$, and  $\widetilde{S}\vert_{\widetilde{S}}$ is a divisor of degree $(-1,-1)$.
Moreover, the nef cone of $X$ is generated by the divisors $A$, $F$, $S+F$ and $A+S+F-E$, and the Mori cone of $X$ is generated by the curves $C_1$, $C_2$, $C_3$, $C_4$ that can be defined as follows:
\begin{itemize}
\item $C_1$ is a fibre of the natural projection $\widetilde{S}\to\mathbf{s}$;
\item $C_2$ is a fibre of the natural projection $\widetilde{A}\to\mathbb{P}^1$;
\item $C_3$ is another ruling of the surface $\widetilde{S}\simeq\mathbb{P}^1\times\mathbb{P}^1$, i.e., $C_3$ is a curve in $|A\vert_{\widetilde{S}}$;
\item $C_4$ is a fibre of the natural projection $E\to Z$.
\end{itemize}
Furthermore, all non-zero intersections of the~divisors $A$, $F$, $S$, $E$ are
$E^3=1$, $F\cdot E^2=-1$, $A\cdot S^2=-1$, $A\cdot F\cdot S=1$, $S\cdot E^2=1$, and the non-zero intersections of the~divisors $A$, $F$, $S$, $E$  with the~curves $C_1$, $C_2$, $C_3$, $C_4$ are $A\cdot C_1=F\cdot C_2=-F\cdot C_3=S\cdot C_3=E\cdot C_1=E\cdot C_2=-E\cdot C_4=1.$

Let $L$ be a nef and big $\mathbb{Q}$-divisor on $X$. Then there are $a,b,c,d\in\mathbb{Q}_{\geqslant 0}$ such that
$$
L\sim_{\mathbb{Q}} aA+bF+c(S+F)+d(A+S+F-E).
$$
We have $L^3=6abc+6abd+3ac^2+6acd+3ad^2+6bcd+3bd^2+3c^2d+6cd^2+2d^3$ and
$$
\frac{(-K_X)\cdot L^{2}}{L^3}=2\frac{2ab+3ac+3ad+2bc+3bd+c^2+5cd+3d^2}{6abc+6abd+3ac^2+6acd+3ad^2+6bcd+3bd^2+3c^2d+6cd^2+2d^3}.
$$
Moreover, we have
$$
\beta_{L}(\widetilde{S})=-\frac{h(a,b,c,d)}{g(a,b,c,d)},
$$
where $g(a,b,c,d)=2(6abc+6abd+3ac^2+6acd+3ad^2+6bcd+3bd^2+3c^2d+6cd^2+2d^3)^2$ and
\begin{multline*}
\quad \quad h(a,b,c,d)=12a^2bc^3+36a^2bc^2d+36a^2bcd^2+12a^2bd^3+18ab^2c^2d+\\
+18ab^2cd^2+12ab^2d^3+36abc^3d+90abc^2d^2+72abcd^3+24abd^4+12ac^3d^2+18ac^2d^3+\\
+9acd^4+3ad^5-6b^2cd^3+12bc^3d^2+18bc^2d^3+6bcd^4+3bd^5+8c^3d^3+9c^2d^4+3cd^5+d^6.\quad \quad \quad
\end{multline*}
Furthermore, we have
$$
\beta_{L}(E)=
\left\{\aligned
&-\frac{f_1(a,b,c,d)}{g(a,b,c,d)}\ \text{if $a\leqslant c$},\\
&-\frac{f_2(a,b,c,d)}{g(a,b,c,d)}\ \text{if $a\geqslant c$},
\endaligned
\right.
$$
where
\begin{multline*}
f_1(a,b,c,d)=36a^2b^2cd+24a^2bc^3+90a^2bc^2d+90a^2bcd^2+12a^2bd^3
+18a^2c^2d^2+18a^2cd^3\\+36ab^2c^2d+72ab^2cd^2+12ab^2d^3+72abc^3d+198abc^2d^2
+156abcd^3+24abd^4+24ac^3d^2+\\60ac^2d^3+39acd^4+3ad^5+12b^2cd^3+24bc^3d^2
+48bc^2d^3+36bcd^4+3bd^5+16c^3d^3+27c^2d^4+15cd^5+d^6,\quad \quad \quad \quad \quad \quad
\end{multline*}
and
\begin{multline*}
\quad \quad f_2(a,b,c,d)=12a^5b+18a^5c+18a^5d-12a^4bc+42a^4bd-30a^4c^2+30a^4cd+\\
+54a^4d^2-24a^3bc^2-84a^3bcd+36a^3bd^2-12a^3c^3-156a^3c^2d-84a^3cd^2+36a^3d^3+36a^2b^2cd+\\
+48a^2bc^3+90a^2bc^2d-18a^2bcd^2+12a^2bd^3+36a^2c^4+108a^2c^3d-54a^2c^2d^2-90a^2cd^3+36ab^2c^2d+\\
+72ab^2cd^2+12ab^2d^3+12abc^4+ 156abc^3d+306abc^2d^2+156abcd^3+24abd^4-6ac^5+42ac^4d+204ac^3d^2+\\
+168ac^2d^3+39acd^4+3ad^5+12b^2cd^3-12bc^5-42bc^4d-12bc^3d^2+48bc^2d^3+36bcd^4+3bd^5-\\
-6c^6-42c^5d-78c^4d^2-20c^3d^3+27c^2d^4+15cd^5+d^6.\quad \quad \quad \quad \quad \quad
\end{multline*}
Note that $f_1\ne f_2$, so $\beta_L(E)$ is not a rational function in $a$, $b$, $c$, $d$. Moreover, if $L$ is ample, then $(X,L)$ is K-unstable  by the following result.

\begin{lemma}
\label{lemma:4-11}
If $L$ is ample, then $\beta_{L}(E)<0$ or $\beta_L(\widetilde{S})<0$.
\end{lemma}

\begin{proof}
Suppose that $L$ is ample. Then $a$, $b$, $c$, $d$ are positive.
If $a\leqslant c$, then $f_1(a,b,c,d)>0$, since all coefficients of the polynomial $f_1$ are positive.
Hence, if $a\leqslant c$, then $\beta_{L}(E)<0$. If $a\geqslant c$, then
\begin{multline*}
h(a,b,c,d)=18b^2c^3d+18b^2c^2d^2+18b^2c^2d\delta+6b^2cd^3+18b^2cd^2\delta+12b^2d^3\delta+\\
+12bc^5+72bc^4d+24bc^4\delta+138bc^3d^2+108bc^3d\delta+12bc^3\delta^2+102bc^2d^3+162bc^2d^2\delta+\\ +36bc^2d\delta^2+30bcd^4+96bcd^3\delta+36bcd^2\delta^2+3bd^5+24bd^4\delta+12bd^3\delta^2+12c^4d^2+\\
+26c^3d^3+12c^3d^2\delta+18c^2d^4+18c^2d^3\delta+6cd^5+9cd^4\delta+d^6+3d^5\delta,\quad \quad \quad \quad
\end{multline*}
where $\delta=a-c\geqslant 0$. Thus, if $a\geqslant c$, then $h(a,b,c,d)>0$, so $\beta_L(\widetilde{S})<0$.
\end{proof}

Let us explain the computations of $\beta_L(\widetilde{S})$ and $\beta_{L}(E)$. We start with $\beta_{L}(\widetilde{S})$. Take $u\in\mathbb{R}_{\geqslant 0}$ and consider the divisor $L-u\widetilde{S}$. If $u\leqslant d$, then $L-u\tilde S$ is nef, so
$$
\mathrm{vol}(L-u\widetilde{S})=6abc+6abd-6abu+3ac^2+6acd+3ad^2-3au^2+6bcd+3bd^2-3bu^2+3c^2d+6cd^2+2d^3-2u^3.
$$
Similarly, if $u\in[d,c+d]$, then the negative part of the Zariski decomposition of the divisor $L-u\widetilde{S}$ is $(u-d)E$, so its positive part is $(a+d)A+(b+c+d)F+(c+d-u)S$, which gives
$$
\mathrm{vol}(L-u\widetilde{S})=3(c+d-u)(c+2b+d+u)(a+d).
$$
Finally, if $u>c+d$, then $L-u\widetilde{S}$ is not pseudoeffective, so $\mathrm{vol}(L-u\widetilde{S})=0$. Now, using \eqref{equation:beta}, we see that $\beta_{L}(\widetilde{S})=-h/g$ as claimed. Now, we show how to compute $\beta_{L}(E)$. As above, we take $u\in\mathbb{R}_{\geqslant 0}$ and consider the divisor $L-uE$. If $u\leqslant \mathrm{min}(a,c)$, then $L-uE$ is nef, so
$$
\mathrm{vol}(L-uE)=6abc+6abd+3ac^2+6acd+3ad^2+6bcd+3bd^2-6bdu-3bu^2+3c^2d+6cd^2+2d^3-3d^2u-3du^2-u^3.
$$
If $a\leqslant c$ and $u\in[a,c]$, the negative part of the Zariski decomposition of  $L-uE$ is $(u-a)\widetilde{S}$,
so its positive part is $(a+d)A+(b+c+d)F+(a+c+d-u)S-(d+a)E$, which gives
$$
\mathrm{vol}(L-uE)=(a+d)(-a^2+3ab+ad+3au+6bc+3bd-6bu+3c^2+6cd+2d^2-3du-3u^2).
$$
Likewise, if $a\leqslant c$ and $u\in[c,a+c+d]$,
then the negative part of the Zariski decomposition of the divisor $L-uE$ is $(u-a)\widetilde{S}+(u-c)\widetilde{A}$,
so its positive part is $(a+c+d-u)(A+S-E)+(b+c+d)F$, which gives
$$
\mathrm{vol}(L-uE)=(a+c+d-u)^2(-a+3b+2c+2d+u).
$$
If $a\leqslant c$ and $u>a+c+d$, then $L-uE$ is not pseudoeffective, so $\mathrm{vol}(L-uE)=0$. Thus, if $a\leqslant c$, then it follows from \eqref{equation:beta} that $\beta_L(E)=-f_1/g$ as claimed. Similarly, if $a\geqslant c$ and $u\in[c,a]$, then the negative part of the Zariski decomposition of the divisor $L-uE$ is $(u-c)\widetilde{A}$, so its positive part is $(a+c+d-u)A+(b+c+d)F+(c+d)S-(d+c)E$, which gives
$$
\mathrm{vol}(L-uE)=(c+d)(6ab+3ac+3ad+3bc+3bd-6bu+2c^2+4cd-3cu+2d^2-3du).
$$
As above, if $a\geqslant c$ and $u\in[a,a+c+d]$, then the negative part of the Zariski decomposition of the divisor $L-uE$ is $(u-a)\widetilde{S}+(u-c)\widetilde{A}$, so its positive part is $(a+c+d-u)(A+S-E)+(b+c+d)F$, which gives
$$
\mathrm{vol}(L-uE)=(a+c+d-u)^2(-a+3b+2c+2d+u).
$$
Finally, if $a\geqslant c$ and $u>a+c+d$, then $L-uE$ is not pseudoeffective, so $\mathrm{vol}(L-uE)=0$. Thus, if $a\geqslant c$, then it follows from \eqref{equation:beta} that $\beta_{L}(E)=-f_2/g$ as claimed.

\subsection{Family \textnumero 4.12}
\label{subsection:4-12}

Let $P_1$ and $P_2$ be two points in $\mathbb{P}^3$, let $\ell$ be a line in $\mathbb{P}^3$ that passes through $P_1$ and $P_2$, let $\phi\colon Y\to\mathbb{P}^3$ be the blow-up of the points $P_1$ and $P_2$, let $\pi\colon X\to Y$ be the blow-up of the strict transform of the line $\ell$, let $E$ be the $\pi$-exceptional surface, let $F_1$ and $F_2$ be the strict transforms on $X$ of the $\phi$-exceptional surfaces, and let $H=(\phi\circ\pi)^*(\mathcal{O}_{\mathcal{P}^3}(1))$. Then $X$ is the unique smooth Fano threefold in Family \textnumero 4.12.

The nef cone of $X$ is generated by the divisors $H$, $H-F_1$, $H-F_2$, $H-E-F_1-F_2$.
To describe the Mori cone of $X$, note that $F_1\simeq F_2\simeq\mathbb{F}_1$, $E\simeq\mathbb{P}^1\times\mathbb{P}^1$
and $E\vert_{E}$ is a divisor of degree $(-1,-1)$.
Let $C_1$ and $C_2$ be fibres of the natural projections $F_1\to\mathbb{P}^1$ and $F_2\to\mathbb{P}^1$, respectively,
let $C_3$ be the fibre of the natural projection $E\to\ell$, and let $C_4$ be the other ruling of the surface $E\simeq\mathbb{P}^1\times\mathbb{P}^1$.
Then the curves $C_1$, $C_2$, $C_3$, $C_4$ generate the Mori cone of $X$.

Let $L$ be a nef and big $\mathbb{Q}$-divisor on $X$. Then there are $a,b,c,d\in\mathbb{Q}_{\geqslant 0}$ such that
$$
L\sim_{\mathbb{Q}} aH+b(H-F_1)+c(H-F_2)+d(H-E-F_1-F_2).
$$
We have $L^3=a^3+3a^2b+3a^2c+3a^2d+3ab^2+6abc+6abd+3ac^2+6acd+3b^2c+3bc^2+6bcd$ and
$$
\frac{(-K_X)\cdot L^{2}}{L^3}=2\frac{2a^2+4ab+4ac+3ad+b^2+4bc+2bd+c^2+2cd}{a^3+3a^2b+3a^2c+3a^2d+3ab^2+6abc+6abd+3ac^2+6acd+3b^2c+3bc^2+6bcd}.
$$
Moreover, we have
$$
\beta_{L}(E)=-\frac{f(a,b,c,d)}{g(a,b,c,d)},
$$
where $g=(a^3+3a^2b+3a^2c+3a^2d+3ab^2+6abc+6abd+3ac^2+6acd+3b^2c+3bc^2+6bcd)^2$ and
\begin{multline*}
\quad \quad f=3a^5d+3a^4b^2+18a^4bd+3a^4c^2+18a^4cd+4a^3b^3+12a^3b^2c+\\
+30a^3b^2d+12a^3bc^2+84a^3bcd+6a^3bd^2+4a^3c^3+30a^3c^2d+6a^3cd^2+12a^2b^3c+12a^2b^3d+\\
+36a^2b^2c^2+108a^2b^2cd+12a^2bc^3+108a^2bc^2d+36a^2bcd^2+12a^2c^3d+24ab^3c^2+36ab^3cd+\\
+24ab^2c^3+90ab^2c^2d+18ab^2cd^2+36abc^3d+18abc^2d^2+12b^3c^3+12b^3c^2d+12b^2c^3d.\quad \quad \quad
\end{multline*}
Thus, if $L$ is ample, then $a$, $b$, $c$, $d$ are positive, so $\beta_L(E)<0$, and $(X,L)$ is K-unstable.

Let us explain how we computed $\beta_{L}(E)$. Since we have an obvious symmetry swapping divisors $F_1$ and $F_2$, we may further assume that $b\leqslant c$, which does not affect our computations. Take $u\in\mathbb{R}_{\geqslant 0}$ and consider the divisor $L-uE$. If $u\leqslant b$, then $L-uE$ is nef, so
\begin{multline*}
\quad \quad \mathrm{vol}(L-uE)=a^3+3a^2b+3a^2c+3a^2d+3ab^2+6abc+\\
+6abd+3ac^2+6acd-6adu-3au^2+3b^2c+3bc^2+6bcd-3du^2-2u^3.\quad \quad \quad \quad
\end{multline*}
Moreover, if $u\in[b,c]$, then the negative part of the Zariski decomposition of the divisor $L-uE$ is $(u - b)F_1$, so its positive part is $(a+b+c+d)H-(d+u)F_1-(c+d)F_2-(d+u)E$, which gives
$$
\mathrm{vol}(L-uE)=(a+b)(a^2+2ab+3ac+3ad+b^2+3bc+3bd+3c^2+6cd-6du-3u^2).
$$
Furthermore, if $u\in[c,a+b+c]$, then the negative part of the Zariski decomposition of the divisor $L-uE$ is $(u-b)F_1+(u-c)F_2$, so its positive part is $(a+b+c+d)H-(u+d)(E+F_1+F_2)$, which gives
$$
\mathrm{vol}(L-uE)=(a+b+c-u)^2(a+b+c+3d+2u).
$$
Finally, if $u>a+b+c$, then $L-uE$ is not pseudoeffective, so $\mathrm{vol}(L-uE)=0$. Thus, using \eqref{equation:beta} we obtain that $\beta_{L}(E)=-f/g$ as claimed.

\section{Birational reduction and auxiliary extractions}\label{section:birational-reduction}

The following result shows how the invariant $\beta_L(\mathscr{F})$ behaves when the $\mathbb{Q}$-divisor $L$
approaches the border of the nef cone of $X$. We will use this in all remaining cases to prove K-instability using computations of the $\beta$-invariant in higher Picard rank cases in Section~\ref{section:beta}.

\begin{proposition}
\label{proposition:Paolo-2}
Let $X$ be a normal projective variety with $\mathbb{Q}$-Gorenstein singularities.
Suppose that there exists a $K_X$-non-positive projective birational morphism $f\colon X\to Y$ such that $Y$ is normal and has $\mathbb{Q}$-Gorenstein  singularities,
and $L\sim_{\mathbb{Q}}f^*(H)$ for some nef and big $\mathbb{Q}$-divisor $H$ on the variety  $Y$.
Then, for every prime divisor $\mathscr{F}$ over $X$, we have 
$\beta_L(\mathscr{F})=\beta_H(\mathscr{F})$.
\end{proposition}

\begin{proof}
First note that, for any $u\in \mathbb{R}$, we have ${\rm vol}(L-uF)={\rm vol}(H-uF)$ and
$$
\frac{-K_X\cdot L^{n-1}}{L^n}=\frac{-K_Y\cdot H^{n-1}}{H^n},
$$
where $n$ is the dimension of $X$.
Moreover, since $f$ is $K_X$-non-positive, we have $K_X\sim_{\mathbb{Q}} f^*(K_Y)+E$ for some $f$-exceptional $\mathbb Q$-divisor $E\geqslant 0$. Then
$$
{\rm ord}_\mathscr{F}(E)= A_Y(\mathscr{F})-A_X(\mathscr{F}).
$$
Furthermore, if $\tau$ is the pseudo-effective threshold of the divisor $L$ with respect to $\mathscr{F}$, then $\tau$ is also the pseudo-effective threshold of $H$ with respect to $\mathscr{F}$. Thus, it is enough to show that
$$\frac 1 {L^n}\int\limits_0^\tau  \frac {\partial {\rm vol}(L+tK_X-u\mathscr{F}) }{\partial t}\Big\vert_{t=0}du
=\frac 1 {H^n}\int\limits_0^\tau \frac {\partial {\rm vol}(H+tK_Y-u\mathscr{F}) }{\partial t}\Big\vert_{t=0}du + {\rm ord}_\mathscr{F}(E).
$$
Since the volume is a $C^1$ function in the interior of the big cone, we have:
$$\begin{aligned}
\frac 1 {L^n}\int\limits_0^\tau  & \frac {\partial {\rm vol}(L+tK_X-u\mathscr{F}) }{\partial t}\Big\vert_{t=0}du= \frac 1 {H^n}\int\limits_0^\tau \frac {\partial {\rm vol}(f^*(H)-u\mathscr{F} + t(f^*(K_Y)+E)) }{\partial t}\Big\vert_{t=0}du\\
&=\frac 1 {H^n}\left(\int\limits_0^\tau \frac {\partial {\rm vol}(H+tK_Y-u\mathscr{F})  }{\partial t}\Big\vert_{t=0}du+
\int\limits_0^\tau \frac {\partial {\rm vol}(f^*(H)+tE-u\mathscr{F})  }{\partial t}\Big\vert_{t=0}du
\right).
\end{aligned}
$$
Thus, it remains to show that
$$\int\limits_0^\tau \frac {\partial {\rm vol}(f^*(H)+tE-u\mathscr{F})  }{\partial t}\Big\vert_{t=0}du=H^n\cdot {\rm ord}_\mathscr{F}(E).
$$
Let $g\colon Z\to X$ be a birational morphism such that $\mathscr{F}$ is a divisor on $Z$ and let $h=f\circ g\colon Z\to Y$. Then we may write $G=g^*(E)-{\rm ord}_\mathscr{F}(E)\cdot \mathscr{F}$, where $G\geqslant 0$ is a $h$-exceptional $\mathbb Q$-divisor, whose support does not contain $\mathscr{F}$.
We then have
$$
\begin{aligned}
{\rm vol}(f^*(H)+tE-u\mathscr{F})&=  {\rm vol}(h^*(H)+tG + (t\cdot {\rm ord}_\mathscr{F}(E)-u)\mathscr{F})\\
&= {\rm vol}(h^*(H)+ (t\cdot {\rm ord}_\mathscr{F}(E)-u)\mathscr{F})
\end{aligned},
$$
where the last equality follows from the negativity lemma and the fact that $G\geqslant 0$.
Thus, we have
$$\frac {\partial {\rm vol}(f^*(H)+tE-u\mathscr{F})  }{\partial t}\Big\vert_{t=0}=-{\rm ord}_\mathscr{F}(E) \cdot \frac {\partial {\rm vol}(f^*(H)-u\mathscr{F})  }{\partial u}.
$$
Thus,
$$\begin{aligned}
\int\limits_0^\tau \frac {\partial {\rm vol}(f^*(H)+tE-u\mathscr{F})  }{\partial t}\Big\vert_{t=0}du&= -{\rm ord}_\mathscr{F}(E) \cdot \int\limits_0^\tau \frac {\partial {\rm vol}(f^*(H)-u\mathscr{F})  }{\partial u}du\\
&=  -{\rm ord}_\mathscr{F}(E)\cdot ({\rm vol}(f^*(H)-\tau \mathscr{F})-{\rm vol}(f^*(H)))\\
&= H^n\cdot {\rm ord}_\mathscr{F}(E).
\end{aligned}
$$
Thus, our claim follows.
\end{proof}

Let us show how to compute $\beta_L(\mathscr{F})$ in one case.

\begin{example}[{cf.  \cite{CheltsovMartinez-Garcia}}]
\label{example:F1-dP7}
Suppose that $X$ is the smooth del Pezzo surface of degree $7$ and let $L$ be a nef and big $\mathbb Q$-divisor on $X$. Then $X$ contains exactly three $(-1)$-curves, which we denote by $E$, $F_1$, $F_2$.
We may assume that the curves $F_1$ and $F_2$ are disjoint, and $E\cdot F_1=E\cdot F_2=1$.
Then $-K_X\sim 3E+2(F_1+F_2)$ and
$$
L\sim_{\mathbb{Q}} a_1(E+F_1)+a_2(E+F_2)+b(E+F_1+F_2)
$$
for some non-negative rational numbers $a_1$, $a_2$, $b$. Since $L$ is assumed to be big, either $b>0$, or both $a_1>0$ and $a_2>0$. Note that $L$ is ample if and only if $b>0$, $a_1>0$, $a_2>0$. If $a_1=a_2=b=1$, then $L\sim -K_X$. We have $L^2=2a_1a_2+2a_1b+2a_2b+b^2$ and
$$
\mu=\frac{2a_1+2a_2+3b}{2a_1a_2+2a_1b+2a_2b+b^2}.
$$
Let us compute $\beta_L(E)$. Without loss of generality, we may assume that $a_1\leqslant a_2$. Fix $u\in\mathbb{R}_{\geqslant 0}$. Then
$$
L-uE\sim_{\mathbb{R}} (a_1+a_2+b-u)E+(a_1+b)F_1+(a_2+b)F_2,
$$
which implies that $L-uE$ is pseudoeffective $\iff$ $u\leqslant a_1+a_2+b$.
Moreover, if $u\leqslant a_1+a_2+b$, then the Zariski decomposition of the divisor $L-uE$ can be described as follows:
\begin{itemize}
\item if $u\leqslant a_1$, then the divisor $L-uE$ is nef;
\item if $a_1\leqslant u\leqslant a_2$, then the positive part of the Zariski decomposition is
$$
(a_1+a_2+b-u)(E+F_2)+(a_1+b)F_1,
$$
and the negative part is $(u-a_1)F_2$;
\item if $a_2\leqslant u\leqslant a_1+a_2+b$, then the positive part of the Zariski decomposition is
$$
(a_1+a_2+b-u)(E+F_1+F_2),
$$
and the negative part is $(u-a_2)F_1+(u-a_1)F_2$.
\end{itemize}
This gives
$$
\mathrm{vol}\big(L-uE\big)=
\left\{\aligned
&\big(L-uE\big)^2\ \text{if $0\leqslant u\leqslant a_1$}, \\
&\big((a_1+a_2+b-u)(E+F_2)+(a_1+b)F_1\big)^2\ \text{if $a_1<u\leqslant a_2$},\\
&\big((a_1+a_2+b-u)(E+F_1+F_2)\big)^2\ \text{if $a_2<u\leqslant a_1+a_2+b$},\\
&0\ \text{if $a_1+a_2+b<u$}.
\endaligned
\right.
$$
Thus, we compute
$$
\mathrm{vol}\big(L-uE\big)=
\left\{\aligned
&2a_1a_2+2a_1b+2a_2b+b^2-2bu-u^2\ \text{if $0\leqslant u\leqslant a_1$}, \\
&(a_1+b)(a_1+2a_2+b-2u)\ \text{if $a_1<u\leqslant a_2$},\\
&(a_1+a_2+b-u)^2\ \text{if $a_2<u\leqslant a_1+a_2+b$},\\
&0\ \text{if $a_1+a_2+b<u$}.
\endaligned
\right.
$$
This gives
$$
\frac{\partial \mathrm{vol}\big(L+tK_X-uE\big)}{\partial t}\Big\vert_{t=0}=
\left\{\aligned
&2u-4a_1-4a_2-6b\ \text{if $0\leqslant u\leqslant a_1$},\\
&4u-6a_1-4a_2-6b\ \text{if $a_1<u\leqslant a_2$},\\
&6u-6a_1-6a_2-6b\ \text{if $a_2<u\leqslant a_1+a_2+b$},\\
&0\ \text{if $a_1+a_2+b<u$}.
\endaligned
\right.
$$
Now, integrating, we obtain
$$
\beta_L(E)=-\frac{2b(3a_1^2a_2+3a_1a_2^2+6a_1a_2b+a_1b^2+a_2b^2)}{3(2a_1a_2+2a_1b+2a_2b+b^2)^2},
$$
so $(X,L)$ is K-unstable if $L$ is ample. Let $f\colon X\to Y$ be the contraction of the curve $F_2$. Then $Y\simeq\mathbb{F}_1$. Moreover, if $a_1=0$, then $L\sim_{\mathbb{Q}}f^*(H)$, where $H$ is a big and nef divisor on $Y$. In this case, it follows from Proposition~\ref{proposition:Paolo-2} that
$$
\beta_H(E)=-\frac{2a_2b}{3(2a_2+b)^2},
$$
which implies that $(Y,H)$ is K-unstable if $H$ is ample.
\end{example}

\subsection{Family \textnumero 2.30}
\label{subsection:2-30}
Observe that $\widetilde{\mathbb{P}}^3$ in Section\,\ref{subsection:3-23} is the smooth Fano threefold in Family~\textnumero 2.30. Moreover, if $b=0$, then $L\sim_{\mathbb{Q}} \varphi^*(H)$,
where $H$ is a big and nef divisor on $\widetilde{\mathbb{P}}^3$ and $L$ is a big and nef divisor as in Subsection\,\ref{subsection:3-23}. In this case, we have
$$
\beta_H\big(\Pi_C\big)=\beta_L(\Pi_C)=-\frac{3a^2c(2a^3+9a^2c+16ac^2+8c^3)}{2(a^3+6a^2c+6ac^2+2c^3)^2}.
$$
by Proposition~\ref{proposition:Paolo-2}. If $H$ is ample, then $a, c>0$, so that $(\widetilde{\mathbb{P}}^3,H)$ is K-unstable as $\beta_H(\Pi_C)<0$.

\subsection{Family \textnumero 2.31}
\label{subsection:2-31}
Observe that $\widehat{Q}$ in Section\,\ref{subsection:3-23} is the smooth Fano threefold in Family \textnumero 2.31. Moreover, if $a=0$, then $L\sim_{\mathbb{Q}} \psi^*(H)$,
where $H$ is a big and nef divisor on $\widehat{Q}$ and $L$ is a big and nef divisor as in Subsection\,\ref{subsection:3-23}. In this case, we have
$$
\beta_H\big(\Pi_C\big)=\beta_L(\Pi_C)=-\frac{bc(4b^2+9bc+6c^2)}{(3b^2+6bc+2c^2)^2}
$$
by Proposition~\ref{proposition:Paolo-2}. Thus, if $H$ is ample, then $b,c>0$, so $\beta_H(\Pi_C)<0$, and $(\widehat{Q},H)$ is K-unstable.

\subsection{Family \textnumero 2.33}
\label{subsection:2-33}
Refer to Section~\ref{subsection:4-12} and
let $f\colon X\to Y$ be the contraction of the face of the Mori cone generated by the rays $\mathbb{R}_{\geqslant 0}[C_1]$ and $\mathbb{R}_{\geqslant 0}[C_2]$. Then $Y$ is the unique smooth Fano threefold in Family \textnumero 2.33, and $f$ contracts both surfaces $F_1$ and $F_2$ to smooth curves. Suppose that $b=c=0$. Then $L\sim_{\mathbb{Q}} f^*(H)$, where $H$ is a big and nef divisor on $Y$ and $L$ is a big and nef  divisor as in Section~\ref{subsection:4-12}. By Proposition~\ref{proposition:Paolo-2}, we have
$$
\beta_H\big(E\big)=\beta_L\big(E\big)=-\frac{3ad}{(a+3d)^2}.
$$
Moreover, if $H$ is ample, then $a$ and $d$ are positive, so $\beta_H(E)<0$, and $(Y,H)$ is K-unstable.

\subsection{Family \textnumero 2.35}
\label{subsection:2-35}
Refer to Section~\ref{subsection:4-12} and let $f\colon X\to Y$ be the contraction of the face of the Mori cone generated by the rays $\mathbb{R}_{\geqslant 0}[C_2]$ and $\mathbb{R}_{\geqslant 0}[C_3]$. Then $Y$ is the unique smooth Fano threefold in Family \textnumero 2.35, and $f$ contracts surfaces $F_2$ and $E$. Suppose that $c=d=0$ in the computations in Section~\ref{subsection:4-12}. Then $L\sim_{\mathbb{Q}} f^*(H)$, where $H$ is a big and nef divisor on $Y$ and $L$ is a  big and nef divisor as in Section~\ref{subsection:4-12}. By Proposition~\ref{proposition:Paolo-2}, we have
$$
\beta_H\big(E\big)=\beta_L\big(E\big)=-\frac{ab^2(3a+4b)}{(a^2+3ab+3b^2)^2}.
$$
Moreover, if $H$ is ample, then $a$ and $b$ are positive, so that $\beta_H(E)<0$, and $(Y,H)$ is K-unstable.

\subsection{Family \textnumero 2.36}
\label{subsection:2-36}
Observe that $\mathbb{P}(\mathcal{O}_{\mathbb{P}^2}\oplus\mathcal{O}_{\mathbb{P}^2}(2))$ appearing in Section\,\ref{subsection:3-29} is the smooth Fano threefold in Family \textnumero 2.36.
Let $a=0$ in the computations in Section\,\ref{subsection:3-29}. Then $L\sim_{\mathbb{Q}} \theta^*(H)$,
where $H$ is a big and nef divisor on $\mathbb{P}(\mathcal{O}_{\mathbb{P}^2}\oplus\mathcal{O}_{\mathbb{P}^2}(2))$ and $L$
is a big and nef divisor  as in Section\,\ref{subsection:3-29}. It follows from  Proposition~\ref{proposition:Paolo-2} that
$$
\beta_H\big(E\big)=\beta_L(E)=-\frac{3c(4b^3+9b^2c+6bc^2+2c^3)}{(3b^2+6bc+4c^2)^2}.
$$
In particular, if the divisor $H$ is ample, then $b$ and $c$ are positive, so $\beta_H(E)<0$, which implies that $(\mathbb{P}(\mathcal{O}_{\mathbb{P}^2}\oplus\mathcal{O}_{\mathbb{P}^2}(2)),H)$ is K-unstable.

\subsection{Family \textnumero 3.26}
\label{subsection:3-26}
We refer to Section\,\ref{subsection:4-9} with the notation therein.
Let $\eta\colon X\to Y$ be the contraction of the extremal ray $\mathbb{R}_{\geqslant 0}[C_2]$.
Then $Y$ is the smooth Fano threefold in Family \textnumero 3.26, and $\eta$ contracts $F_1$ to a smooth curve.
If $c=0$ in the computation in Section\,\ref{subsection:4-9}, then $L\sim_{\mathbb{Q}} \eta^*(H)$,
where $H$ is a big and nef divisor on $Y$ and $L$ is a big and nef divisor as in Section\,\ref{subsection:4-9}. In this case, we have
$$
\beta_H\big(S\big)=\beta_L(S)=-\frac{a(3a^4d+3a^3b^2+15a^3bd+4a^2b^3+42a^2b^2d+42ab^3d+18ab^2d^2+12b^4d+18b^3d^2)}{2(a^3+3a^2b+3a^2d+3ab^2+6abd+3b^2d)^2}
$$
by Proposition~\ref{proposition:Paolo-2}. Thus, if $H$ is ample, then $a$, $b$ and $d$ are positive, so $\beta_H(S)<0$, which implies that $(Y,H)$ is K-unstable.

\subsection{Family \textnumero 3.28}
\label{subsection:3-28}
Recall that $\mathbb{P}^1\times\mathbb{F}_1$ is the unique smooth Fano threefold in Family \textnumero 3.28. Refer to Section~\ref{subsection:4-11} and set
 $d=0$. Then, for the divisor $L$ as in Section~\ref{subsection:4-11}, we have $L\sim_{\mathbb{Q}} \pi^*(H)$,
where $H$ is a big and nef divisor on $\mathbb{P}^1\times\mathbb{F}_1$. By Proposition~\ref{proposition:Paolo-2}, we get
$$
\beta_H\big(\widetilde{S}\big)=\beta_L\big(\widetilde{S}\big)=-\frac{2bc}{3(2b+c)^2}.
$$
Moreover, if $H$ is ample, then $a$, $b$, $c$ are positive, so that $\beta_H(\widetilde{S})<0$, and $(\mathbb{P}^1\times\mathbb{F}_1,H)$ is K-unstable.

\subsection{Family \textnumero 3.31}
\label{subsection:3-31}
Refer to Section~\ref{subsection:4-11} and the notation therein.
Let $f\colon X\to Y$ be the contraction of the extremal ray $\mathbb{R}_{\geqslant 0}[C_2]$.
Then $Y$ is the smooth Fano threefold in Family \textnumero 3.31, and $f$ contracts the surface $\widetilde{A}$ to a smooth curve.
Let $c=0$ in the computations in Section~\ref{subsection:4-11}. Then $L\sim_{\mathbb{Q}} f^*(H)$,
where $H$ is a big and nef divisor on $Y$, and $L$ is the divisor as in Section~\ref{subsection:4-11}. By Proposition~\ref{proposition:Paolo-2}, we have
$$
\beta_H\big(\widetilde{S}\big)=\beta_L\big(\widetilde{S}\big)=-\frac{12a^2bd^3+12ab^2d^3+24abd^4+3ad^5+3bd^5+d^6}{2(6abd+3ad^2+3bd^2+2d^3)^2}.
$$
Moreover, if $H$ is ample, then $a$, $b$, $d$ are positive, so that $\beta_H(\widetilde{S})<0$, and $(Y,H)$ is K-unstable.

\subsection{Family \textnumero 3.30}
\label{subsection:3-30}
Refer to Section~\ref{subsection:4-12} and let $f\colon X\to Y$ be the contraction of the extremal ray $\mathbb{R}_{\geqslant 0}[C_2]$.
Then $Y$ is the smooth Fano threefold in Family \textnumero 3.30, and $f$ contracts $F_2$ to a smooth curve. Let $c=0$ in the computation in Section~\ref{subsection:4-12}. Then $L\sim_{\mathbb{Q}} f^*(H)$,
where $H$ is a big and nef divisor on $Y$ and $L$ is a big and nef  divisor as in Section~\ref{subsection:4-12}. By Proposition~\ref{proposition:Paolo-2}, we have
$$
\beta_H\big(E\big)=\beta_L\big(E\big)=-\frac{3a^3d+3a^2b^2+18a^2bd+4ab^3+30ab^2d+6abd^2+12b^3d}{(a^2+3ab+3ad+3b^2+6bd)^2}.
$$
Moreover, if $H$ is ample, then $a$, $b$, $d$ are positive, so $\beta_H(E)<0$, and $(Y,H)$ is K-unstable.

\end{document}